\documentclass[twoside,a4paper]{amsart}

\usepackage{mathabx}

\usepackage{pdfsync}
\usepackage{enumitem}
\usepackage{tikz}
\usetikzlibrary{calc}

\usepackage{xcolor}
\definecolor{webgreen}{rgb}{0,.5,0}
\definecolor{webbrown}{rgb}{.6,0,0}
\definecolor{RoyalBlue}{cmyk}{1, 0.50, 0, 0}
\usepackage[colorlinks=true, breaklinks=true, urlcolor=webbrown, linkcolor=RoyalBlue, citecolor=webgreen,backref=page]{hyperref}
\usepackage{bbold}

\usepackage{courier} % tt
\normalfont
\usepackage{epsfig, graphicx, subfigure}
\usepackage{amsmath, amssymb}

\newcommand{\T}     {\mathbb{T}}
\newcommand{\N}  {\mathbb{N}}
\newcommand{\D}     {\mathbb{D}}
\newcommand{\R}     {\mathbb{R}}

\newcommand{\Z}     {\mathbb{Z}}
\newcommand{\dist}{\mathrm{dist}}

\newcommand{\supp}{\mathrm{supp}}

\let\Re=\undefined
\DeclareMathOperator{\Re}{Re}
\let\Im=\undefined
\DeclareMathOperator{\Im}{Im}

\def\ge{\geqslant}
\def\le{\leqslant}

\newtheorem{theorem}{Theorem}[section]

\newtheorem{corollary}[theorem]{Corollary}
\newtheorem{lemma}[theorem]{Lemma}

\newtheorem{definition}[theorem]{Definition}

\newtheorem{remark}[theorem]{Remark}

\numberwithin{equation}{section}

\begin{document}

\title[Wave operators for Jacobi matrices ]{
Wave operators for Jacobi matrices    }

\begin{abstract}
We study the wave operators for a Jacobi matrix $J$ whose canonical spectral measure $\rho$ satisfies the  Szeg\H{o} condition. In particular, for $\gamma_s$ the Verblunsky coefficients of the measure associated to $\rho$ by the Szeg\H o mapping, we prove existence and completeness of wave operators under the condition $\lim_{n\to\infty}\log(n)\sum_{s=n}^{2n}|\gamma_s|^2=0$. In the course of proving our main theorem, we establish an estimate on the norm of orthogonal polynomials on the unit circle localized to arcs that may be of independent interest. 
\end{abstract} \vspace{1cm}

%\thanks{
%The work of SD done in the first two sections
%  was supported by the grant NSF-DMS-1764245 and his research on the rest of the paper was supported by the Russian Science Foundation (project RScF-19-71-30004). The work of LM was supported by the grant RTG NSF-DMS-1147523.}

\author[S.A. Denisov]{Sergey A. Denisov}
\address{Department of Mathematics, University of Wisconsin-Madison, 480 Lincoln Dr., Madison, WI 53706, USA}
\email{\href{mailto:denissov@wisc.edu}{denissov@wisc.edu}}

\thanks{
The research of S.D. was supported by the grant NSF-DMS-2450716, Simons Fellowship in Mathematics, Simons Travel Support for Mathematicians Award, and by the Van Vleck Professorship Research Award. He gratefully acknowledges the hospitality of IHES where part of this work was done.
}

\author[G. Young]{Giorgio Young}
\address{Department of Mathematics, University of Wisconsin-Madison, 480 Lincoln Dr., Madison, WI 53706, USA}
\email{\href{mailto:gfyoung@wisc.edu}{gfyoung@wisc.edu}}

\thanks{G.Y.\ acknowledges the support of the National Science Foundation through grant DMS--2303363.}



\maketitle

\setcounter{tocdepth}{3}
\setcounter{tocdepth}{2}

\tableofcontents

%\section*{Contents}
%\begin{enumerate}
%\item Introduction
%\item Preliminary OPUC Results
%\item Existence and Completeness of Wave Operators
%\item Appendix
%\begin{enumerate}[label=\arabic{enumi}.\arabic*.]
%\item Auxiliary Proofs and an OPUC Estimate
%\item The Weak Nonlinear Carleson Condition
%\end{enumerate}
%\end{enumerate}

\section{Introduction}

We consider semi-infinite Jacobi matrices 
\begin{align*}
J = \begin{pmatrix}
v_1& b_1&0&&&&\\
b_1& v_2&b_2&0&&&\\
0&b_2&v_3&b_3&0&&\\
&\ddots&\ddots&\ddots&\ddots&\ddots\\
\end{pmatrix}
\end{align*}
in the bounded and self-adjoint setting, where the Jacobi coefficients $\{ b_n\}$ and $\{v_n\}$ satisfy $b_n>0, v_n\in \R$ for all $n\in \N$, and $\{b_n\}, \{v_n\}\in \ell^\infty(\N)$.
Under these conditions, $J$ defines a bounded self-adjoint operator $J:\ell^2(\N)\to\ell^2(\N)$. Through the special case $b_n=1/2$ for all $n\in \N$, one recovers discrete Schr\"odinger operators, which serve as Hamiltonians for lattice models in quantum mechanics. Jacobi matrices arise as a natural generalization of these operators, particularly from the perspective of inverse spectral theory discussed below. Our main interest in this work is the dynamics generated by $J$ through the discrete time-dependent Schr\"odinger equation:
\begin{align*}
-i\partial_t\psi=J\psi,\quad \psi(0)=f\in \ell^2(\N), 
\end{align*}
whose solution may be written in terms of the one-parameter unitary group as $\psi(t)=e^{itJ}f$.

By Favard's theorem of classical analysis, there is a bijective correspondence between bounded semi-infinite Jacobi matrices $J$ and compactly and infinitely supported (in the sense of cardinality) probability measures on $\R$. In the forward direction of this correspondence, one maps $J$ to its canonical spectral measure, $\rho$. In this work, we consider Jacobi matrices whose corresponding measures $\rho$ are supported on $[-1,1]$, and are in the Szeg\H o class: given the decomposition of $\rho$ with respect to the equilibrium measure $\omega$ for the interval $[-1,1]$,
\begin{align*}
d\rho=\rho' d\omega+\rho_s,\,\,d\omega(x)=\pi^{-1}(1-x^2)^{-\frac 12}\mathbb{1}_{[-1,1]}(x)dx,
\end{align*}
we assume that $\rho$ satisfies the Szeg\H o condition:
\begin{align}\label{eq:Szego}
\int\limits_\R \log \rho'd\omega>-\infty\,.
\end{align}

Despite the inverse direction of this correspondence passing through the orthogonal polynomials on the real line (OPRL), much of our work here makes use of orthogonal polynomials on the unit circle (OPUC). In particular, exploiting a well-known relationship between OPRL with respect to measures supported on the interval $[-1,1]$, and OPUC with respect to a related measure, we may work on the circle and apply harmonic analysis techniques. Thus, our first result is most naturally stated in the OPUC terminology, and we introduce it now. We refer the reader to \cite{SimonOPUC1,SimonOPUC2,simon2010szegHo,Szego} for further background on the brief description that follows.

Let $\sigma$ be an infinitely supported probability measure on the unit circle $\T$. Denote the orthonormal polynomials by $\phi_n(z)$ and the monic orthogonal polynomials by $\Phi_n(z)$.
Clearly, $\phi_n=\Phi_n/\|\Phi_n\|_{2,\sigma}$. We will use notation  $f^*(z)=z^n\overline{f(\bar z^{-1})}$ for a polynomial $f$ of degree $n$. The polynomials $\{\Phi_n\}$ satisfy recurrence \cite[Theorem 1.5.2, p.~57]{SimonOPUC1}
\begin{equation}\label{opuc1}
\left\{\begin{array}{cc}
\Phi_{n+1}(z)=z\Phi_n(z)-\bar\gamma_n \Phi_n^*(z), &\Phi_0(z)=1,\\
\Phi_{n+1}^*(z)=\Phi_n^*(z)-z\gamma_n\Phi_n(z),&\Phi_0^*(z)=1\,,
\end{array}\right.
\end{equation}
where  $\gamma_n$ are the Verblunsky coefficients corresponding to the measure $\sigma$. The map $\sigma\mapsto \{\gamma_n\}|_{n=0}^\infty\in \D^\infty$ is a bijection, cf. \cite[p.~2]{SimonOPUC1}.

Measures $\rho$ supported on $[-1,1]$ may be related to even measures $\sigma$ on $\T$ through the Szeg\H o mapping $Sz$.  We will write $\sigma=Sz^{-1}(\rho)$ for the measure $\sigma$ on $\T$ that satisfies $d\sigma(z)=d\sigma(1/z)$ and
\begin{align}\label{opuc-oprl}
\int\limits_{\T}f(z)d\sigma(z)=\int\limits_{[-1,1]}f(e^{i\arccos(x)})d\rho(x)
\end{align}
for any $f\in C(\T)$ with $f(z)=f(1/z), z\in \T$ \cite[p.~31, eq.~(1.9.5)]{simon2010szegHo}. Throughout the text, $\sigma$ will refer to the measure related to $\rho$ in this way. For such measures, all $\gamma_n$ are real, and  they may be related to the Jacobi coefficients by means of the direct Geronimus relations \cite[p.~33]{simon2010szegHo}
\begin{align}\label{geronimus}
\begin{split}
b_{k+1}&=\frac 12 \left( (1-\gamma_{2k-1})(1-\gamma_{2k}^2)(1+\gamma_{2k+1})    \right)^{\frac 12},\\
v_{k+1}&=-\frac 12 (\gamma_{2k-2}(1+\gamma_{2k-1})-\gamma_{2k}(1-\gamma_{2k-1}))\,,
\end{split}
\end{align}
with the convention $\gamma_{-1}:= -1$. If we denote the polynomials orthonormal with respect to $\rho$ by $p_n(x)$, then \cite[p.~32, eq.~(1.9.12)]{simon2010szegHo}
\begin{align}\label{pnchin0}
p_n((z+1/z)/2)=(2(1-{\gamma}_{2n-1}))^{-1/2}z^{-n}(\phi_{2n}(z)+\phi_{2n}^*(z)), \, z\in \T\,.
\end{align}

Szeg\H o's theorem, cf. \cite[Theorem 1.8.6]{simon2010szegHo}, is a cornerstone of the theory of OPUC; for the decomposition of $\sigma$ with respect to the normalized Lebesgue measure $dm(z)$ on $\T$, $d\sigma(z)=\sigma'(z)dm(z)+\sigma_s(z)$, it states
\begin{align*}
\prod_{n=0}^\infty (1-|\gamma_n|^2)=\exp\left(\int\limits_{\T}\log(\sigma'(z))dm(z) \right)\,,
\end{align*}
which in turn implies
\begin{align*}
\{ \gamma_n\}\in \ell^2(\Z^+) \iff \int\limits_{\T} \log(\sigma'(z))dm(z)>-\infty.
\end{align*}
With this in hand, the identity \eqref{opuc-oprl} implies that the Szeg\H o condition \eqref{eq:Szego} on $J$ is equivalent to the corresponding Verblunsky parameters satisfying $\{\gamma_n\}\in \ell^2(\Z^+)$. Denoting by $J_0$ the free operator with Jacobi coefficients $b_n=\frac 12, v_n=0$ for all $n$, we may combine this observation with the Geronimus relations to find that the condition \eqref{eq:Szego} implies that $J-J_0$ is a Hilbert-Schmidt operator. Thus, in addition to being a qualitative condition on the size and distribution of the absolutely continuous part of the measure $\rho$, this condition allows for $J$ to be viewed as a perturbation of the free operator. 

Both of these perspectives offer a first motivation for the  study of the evolution generated by $J$, particularly in comparison to the free evolution $J_0$. Motivated by classical results of \cite{AmreinGeorgescu,Ruelle} and already given the condition \eqref{eq:Szego}, one may expect the typical state to spread out over $\N$, behaving as in the free evolution in a qualitative sense, meanwhile the perturbative perspective suggests a stronger and more direct comparison may hold between the two evolutions. Such a comparison is often encoded through the existence and completeness of wave operators $\Omega_{\pm}$, a by now well-studied and classical topic. We refer the reader to the foundational scattering theory papers \cite{Agmon1975,BirmanKrein1962,Hormander1976}, as well as the textbook treatments of the subject found in \cite{Kato1995,LaxPhillips1989,ReedSimon3,Yafaev1992}.

When they exist, wave operators are defined as the strong $\ell^2(\N)$-limits
\begin{align}\label{eq:WO}
\lim_{T\to\pm \infty} e^{-iTJ}e^{iTJ_0}=:\Omega_{\pm}.
\end{align}
We denote by $\ell^2_{\rm ac}(\N)$ the absolutely continuous subspace for $J$, or set of vectors whose spectral measure is absolutely continuous, cf. \cite{LukicSpectralTheory,reed1972methods}:
\begin{align*}
\ell^2_{\rm ac}(\N):=\{\psi\in\ell^2(\N):\mu_{\psi} \text{ is a.c.} \},
\end{align*}
where $\mu_\psi$ is spectral measure of $\psi$ relative to $J$.

When $\Omega_{\pm}$ exist, the inclusion $\text{Ran}(\Omega_\pm)\subseteq \ell^2_{\rm ac}(\N)$ always holds. Completeness is defined as the equality $\text{Ran}(\Omega_\pm)=\ell^2_{\rm ac}(\N)$, so that $\Omega_\pm$ are unitary maps from $\ell^2(\N)$ onto $\ell^2_{\rm ac}(\N)$. This equality yields the comparison between the two evolutions alluded to above; indeed, once this is established, one has that every state in $\ell^2_{\rm ac}(\N)$ behaves asymptotically freely. Establishing these results under \eqref{eq:Szego}, along with a mild quantitative condition on the tail of the series for $\{ \gamma_n\}$, will be the main goal of our work here. For the reader more familiar with the OPUC literature, we note that this is a significantly more delicate task than existence and completeness of the CMV wave operators, where the free dynamics are given by translation, cf. \cite[Theorem 10.7.9]{SimonOPUC2}.\smallskip

Working under the following additional assumption on the $\gamma_n$ related to $J$ as in the above, 
\begin{equation}\label{loga}
\lim_{n\to\infty} \log n\sum_{s=n}^{2n}|\gamma_s|^2=0,
\end{equation}
we prove:

\begin{theorem}\label{thm:main} Given $\{\gamma_n\}\in \ell^2(\Z^+)$ and \eqref{loga}, the strong limits \eqref{eq:WO} exist.  Moreover, 
\begin{align*}
{\rm Ran}(\Omega_\pm )=\ell^2_{\rm ac}(\N).
\end{align*}
\end{theorem}
For convenience, we will work only with $\Omega_+$, with analogous proofs establishing the result for $\Omega_-$. Motivation for the condition \eqref{loga} may be found in our discussion of the pointwise convergence assumption in Subsection~\ref{app:PCA}. \medskip

The orthonormal polynomials $p_n\in L^2_\rho(\R)$ form the basis for the eigenfunction expansion for $J$, with respect to which the evolution under $J$ is a multiplication operator. 
In proving Theorem~\ref{thm:main}, we make use of the relationship \eqref{pnchin0} between the polynomials orthogonal with respect $\rho$, and those orthogonal with respect to $\sigma$. We can rewrite \eqref{pnchin0} as
\begin{align}\label{pnchin}
p_n(1/2(z+1/z))=(2(1-{\gamma}_{2n-1}))^{-1/2}(s_n(z)+\overline{s_n}(z))=:c_n(s_n(z)+\overline{s_n}(z)), \, z\in \T\,,
\end{align}
where $s_n:=z^{-n}\phi_{2n}$. With this in hand, and making use of the symmetry $\overline{s_n}(z)=s_n(1/z)$, it will be sufficient to control the $J$-evolution of the free wave packet $e^{iTJ_0}f$ expanded in terms of the $s_n(z)$ in the space $L^2_\sigma(\T)$. In this space, the limit \eqref{eq:WO} may be computed explicitly in terms of the Szeg\H o function $D(z)$, as defined in \eqref{eq:Szegofxn} below, and completeness follows from its form.

As further context for the above theorem, we note that in the case when $J-J_0$ is trace class, the Kato-Rosenblum theorem guarantees existence and completeness of $\Omega_\pm$ \cite[Theorem XI.8]{ReedSimon3}. More recent work in the continuum has shown existence and completeness in a suitable sense for modified wave operators in the intermediate regime of potentials in $L^p(\R^+)$ for $1\leq p<2$ \cite{christ2002scattering}, building on results of the same authors in the stationary scattering setting in \cite{ChristKiselev2001}. We refer the reader to \cite{KimKiselev2009} for analogous stationary scattering results in the discrete setting. Meanwhile, the works \cite{kiselev1998modified,Pearson1978} provide discrete and continuum examples of potentials with slower decay producing purely singular continuous spectrum. This implies that $\ell^2_{\rm ac}(\N)=\{ 0\}$, and so precludes the existence of wave operators. From this perspective, and with the observation above on the Hilbert-Schmidt decay of $J-J_0$, one may view our work here as taking place in the transitional, and optimal, regime. \smallskip

For more closely related and motivating work, we turn to previous results in the setting of Dirac operators. There, in analogy to the OPUC/CMV setting mentioned above, the free dynamics are again translation. We refer the reader to \cite{Denisov2004,LM}, establishing existence of wave operators throughout the transitional regime of $L^2(\R^+)$ potentials, as well as \cite{Bessonov2020}, proving the existence and completeness of modified wave operators for more general potentials under the Szeg\H o condition. In particular, this latter work covers the range of potentials in $L^p(\R^+)$ for $1\leq p\leq 2$. Finally, in the work \cite{BessonovDenisov2023}, the previous two authors completely characterized the close connection between the Szeg\H o condition and time-dependent scattering for Dirac operators hinted at in their prior work. In particular, they prove that for $L^1_{\rm loc}(\R^+)$ potentials, the spectral measure is in the Szeg\H o class if and only if the wave operators exist. Moreover, the Szeg\H o condition implies completeness.\smallskip

Using the direct Geronimus relations \eqref{geronimus}, we may also prove a corollary establishing existence and completeness of the wave operators for $J$ under hypotheses that only depend on the coefficients of $J$. We define 
\begin{align*}
\delta\ell_{\log}^2:=\{\{a_{n+1}-a_n\}: \{a_n\}\in \ell^2(\N),\; a_n\in \R, \,\text{and}\;\{a_n\}\;\text{satisfies}\;\eqref{loga}\}.
\end{align*}

We have the following Corollary of Theorem~\ref{thm:main}.
\begin{corollary}\label{cor:main}
Let $J$ be such that 
\begin{align}\label{kl}
\{v_n\}\in \delta\ell_{\log}^2+\ell^1(\N)\;\text{ and}\,\; \{b_n\}=1/2+\delta\ell_{\log}^2+\ell^1(\N),
\end{align}
then the strong limits \eqref{eq:WO} exist, and ${\rm Ran}(\Omega_{\pm})=\ell^2_{\rm ac}(\N)$.
\end{corollary}

%Our final result is a proof of the existence of wave operators in Ces\` aro time average for all Jacobi operators in the Szeg\H o class.
%\begin{theorem}\label{thm:Cesaro}
%If $\{ \gamma_n\}\in \ell^2(\Z^+)$, then the strong limits  
%\[
%\lim_{T\to +\infty }T^{-1}\sum_{t=0}^{T-1} e^{\mp  itJ}e^{\pm itJ_0}
%\]
%exist.
%\end{theorem}
 
This work is organized as follows: in Section~\ref{section2}, we prove some asymptotic statements on weighted sums of $s_n$ along certain  sequences. These are abstractions of sums that arise naturally after the decomposition into wavepackets that takes place in the proof of Theorem~\ref{thm:main}, and are also necessary for the main result of the section, the convergence result Theorem~\ref{lem:PCA}. This theorem is a key input required to prove Theorem~\ref{thm:main}. Section~\ref{mainsection} contains the proof of our main results: Theorem~\ref{thm:main}, as well as Corollary~\ref{cor:main}. Appendix~\ref{app:1} contains some technical proofs of statements in Section~\ref{section2}, as well as an estimate on the norm of the orthonormal polynomials $\phi_n$ localized to arcs, a key ingredient in these proofs. Finally, in Appendix~\ref{app:PCA}, we discuss the pointwise convergence assumption (PCA), a pointwise strengthening of the $L^2$ convergence that holds for measures in the Szeg\H o class. In that section, we prove that the PCA implies the convergence result of Theorem~\ref{lem:PCA}, and hence Theorem~\ref{thm:main}. Since \eqref{loga} is a weakening of a known sufficient condition for the PCA, this serves as further motivation for working under the assumption \eqref{loga}.\medskip

\textbf{Notation}
\begin{itemize}
%\item $J$ is a Jacobi matrix with Jacobi parameters $v_n$ on the diagonal and $b_n$ on the off-diagonal, and $\rho $ is its corresponding measure supported on $[-1,1]$.
%\item $J_0$ is the free Jacobi matrix with parameters $v_n=0$, $b_n=1/2$ corresponding to the measure $d\rho_0(x)=(1-x^2)d\omega(x)$ for $\omega$ is the equilibrium measure on $[-1,1]$.
\item $\Z^+=\{0,1,2,\dots\}$, $\N=\{1,2,3,\dots\}$, and $\R^+=[0,\infty)$.
\item $\rho$ is the canonical spectral measure for the Jacobi matrix $J$, and $\sigma$ is the even measure on $\T$ related to $\rho$ by the Szeg\H o mapping: $\rho=Sz(\sigma)$.
\item $\phi_n, \Phi_n$ for $n\in \Z^+$ are the orthonormal and monic orthogonal polynomials respectively with respect to $\sigma$, with $\{\gamma_n\}$ the Verblunsky parameters.
\item For a degree $n$ polynomial on the circle $f$, $f^*=z^n\overline{f}$. We define $s_n:= z^{-n}\phi_{2n}=z^{n}\overline{\phi_{2n}^*}$.
\item $p_n$, $n\in\Z^+$ are the orthonormal polynomials with respect to $\rho$.
\item $m$ is the normalized arclength measure on the torus $\T$.
\item We will denote the $\kappa$-th roots of unity by $z_j:=\exp(2\pi ij/\kappa)$, for $j\in\{0,\dots,\kappa-1\}$. We refer to the argument as $x_j:=2\pi j/\kappa$. 
\item $\langle \cdot,\cdot\rangle_{\tau}$ is the inner product of two quantities in $L^2_{\tau}(\cdot)$, and $\perp_{\tau}$ denotes orthogonality in $L^2_\tau(\cdot)$. 
\item $\|\cdot\|_{2,\tau}$ is the $L^2_\tau$ norm, and $\|\cdot\|_\infty$ is the $L^\infty$ norm, where the measure will be clear from context.
\item For a measure $\tau$ on $\T$ satisfying the Szeg\H o condition,
\[
\int\limits_\T \log(\tau'(z))dm(z)>-\infty,\,d\tau(z)=\tau'(z)dm(z)+\tau_s(z)
\]
the Szeg\H o function $D$ in the Hardy space $H^2(\D)$ is  \cite[p.~144, eq.~(2.4.2)]{SimonOPUC1}
\begin{align}\label{eq:Szegofxn}
D(z):=\exp\left(\frac{1}{2}\int\limits_\T\frac{w+z}{w-z}\log(\tau'(w))dm(w) \right)\,.
\end{align}
$D$ has nontangential boundary values $m$-almost everywhere so that we may write $D(z)$ when $|z|=1$ as well.
\item For a sequence $f=\{f_n\}\in \ell^2(\Z)$, $\widehat{f}\in L^2_{m}(\T)$ will denote the Fourier transform of $f$:
\[
\widehat{f}(z)=\sum_{n\in\Z}f_nz^n.
\]
For $g\in L^2_m(\T)$, its inverse is given by the sequence $\{g_n\}$ 
\[
g_n=\langle g,z^n\rangle_{m}=\int\limits_{\T}g(z)z^{-n}dm(z).
\]
\item For a function $f\in L^2(\R)$, the Fourier transform is denoted
by
\[
(\mathcal{F}f)(\xi)=(2\pi)^{-1}\int\limits_\R f(x)e^{ -i x\xi}dx
\]
and the inverse Fourier transform is $\mathcal{F}^{-1}$.

\item $\mathbb{1}_E$ is the indicator function for a set $E$, and $|E|$ denotes its Lebesgue measure. 
\item For a sequence $\{a_j\}$ and a number $P$, $a_j^{(d)}=a_j\mathbb{1}_{j\equiv d\mod P}$ for $0\leq d\leq P-1$. 
\item ${\rm Proj}_Sv$  is the orthogonal projection of vector $v$ onto the subspace $S$ of a Hilbert space, and ${\rm Perp}_Sv=v-{\rm Proj}_Sv$ denotes its perpendicular component.
\item For quantities $A$ and $B$ and $r$ a parameter, $A\le_{r}B$ means that there is a nonnegative constant $C_r$ with $A\leq C_r B$. $A\lesssim B$ means that $A\leq CB$ for a constant $C>0$ independent of the relevant variables, so that $A\lesssim 1$ is shorthand for $A$ being uniformly bounded in the relevant parameters. 
\item $A\sim B$ means $B\lesssim A\lesssim B$, and for a parameter $r$, $A\sim_r B$ means $B\le_{r} A\le_{r}B$.
\item For a natural number $\ell$, the symbol $O_\ell(g)$ stands for a function $f$ that satisfies $|f|\le C_\ell g$ with some $\ell$-dependent constant $C_\ell$ and a nonnegative quantity $g$.
\item $\Omega_{\pm}$ are the strong limits \eqref{eq:WO}.
\end{itemize}\bigskip\bigskip

\section{Preliminary OPUC results}\label{section2}

In this section, $\kappa$ will be a large integer and we assume that the sequence of Szeg\H{o} recursion parameters satisfies $\|\{\gamma_n\}\|_2\le \frac 12$.  Recall that $s_j=\phi_{2j}z^{-j}=\overline{\phi_{2j}^*}z^j, z\in \T$ and $\{s_j\}$ is an orthonormal  system in $L^2_\sigma(\T)$. We also have 
\begin{equation}\label{perp1}
s_j\perp_{\sigma} \{z^{j-1},\ldots,z^{-(j-1)}\}.
\end{equation}

\begin{lemma}\label{lem1}

Suppose $|f_{j,\kappa}|\lesssim 1, |\alpha_{j,u}|\lesssim \kappa^{-1}$ for all $0\le j\le \kappa-1, |u|\le\kappa-1$. Consider
\begin{equation}\label{sum0}
A:=\sum_{j=0}^{\kappa-1}\sum_{ |u|\le \kappa-1}f_{j,\kappa}s_{\kappa^2+2j\kappa+u}\alpha_{j,u}
\end{equation}
and 
\begin{align}\label{bdef}
B=\sum_{j=0}^{\kappa-1}f_{j,\kappa}\overline{\phi^*_{2(\kappa^2+2j\kappa)}}\sum_{ |u|\le \kappa-1}z^{\kappa^2+2j\kappa+u}\alpha_{j,u}=\sum_{j=0}^{\kappa-1}f_{j,\kappa}s_{\kappa^2+2j\kappa}\sum_{|u|\leq \kappa-1}\alpha_{j,u}z^u.
\end{align}

Then, 
\begin{equation}\label{step-11}
\|A-B\|_{2,\sigma}\to 0, \quad \text{as}\quad \kappa\to\infty\,.
\end{equation}
\end{lemma}

\begin{proof}
We can write
\[
A=\sum_{j=0}^{\kappa-1} B_{j}+\sum_{j=0}^{\kappa-1} C_j\,,
\]
where
\[
B_j=f_{j,\kappa}s_{\kappa^2+2j\kappa}\sum_{|u|\leq \kappa-1}\alpha_{j,u}z^u
\]
and
\[
C_j:=\sum_{ |u|\le \kappa-1}f_{j,\kappa}\Bigl(\overline{\phi^*_{2(\kappa^2+2j\kappa+u)}}-\overline{\phi^*_{2(\kappa^2+2j\kappa)}}\Bigr)z^{\kappa^2+2j\kappa+u}\alpha_{j,u}\,.
\]
To finish the proof of the lemma, we only need to show that
\[
\left\|\sum_{j=0}^{\kappa-1} C_j\right\|_{2,\sigma}\to 0, \quad \kappa\to\infty\,.
\]
To this end, we write
\[
C_j=\sum_{ |u|\le \kappa-1}f_{j,\kappa}(s_{\kappa^2+2j\kappa+u}-s_{\kappa^2+2j\kappa}z^u)\alpha_{j,u}\,.
\]
Notice that $C_j\perp_\sigma C_\ell$ for $|j-\ell|\ge1$ by orthogonality property  \eqref{perp1} and by checking the degrees of polynomials involved.
So,
\[
\|\sum_{j=0}^{\kappa-1} C_j\|_{2,\sigma}^2=\sum_{j=0}^{\kappa-1}\left\|C_j\right\|_{2,\sigma}^2\,.
\]
Recall the OPUC recurrence \eqref{opuc1} and the connection between the monic orthogonal and orthonormal polynomials  \cite[eq.~ (1.5.1), p.~55]{SimonOPUC1}
\begin{equation}\label{opuc2}
\phi_n=\nu_n\Phi_n,\quad \nu_n:=\prod_{0\le s\le n-1}\rho^{-1}_s, \quad \rho_s:=(1-|\gamma_s|^2)^{\frac 12}, \quad \nu_\infty:=\prod_{0\le s\le \infty}\rho^{-1}_s\,.
\end{equation}
We write
\[
\phi^*_{2(\kappa^2+2j\kappa+u)}-\phi^*_{2(\kappa^2+2j\kappa)}=\nu_{2(\kappa^2+2j\kappa+u)}\Phi^*_{2(\kappa^2+2j\kappa+u)}-\nu_{2(\kappa^2+2j\kappa)}\Phi^*_{2(\kappa^2+2j\kappa)}
\]
\[
=(\nu_{2(\kappa^2+2j\kappa+u)}-\nu_{2(\kappa^2+2j\kappa)})    \Phi^*_{2(\kappa^2+2j\kappa+u)}+\nu_{2(\kappa^2+2j\kappa)}(\Phi^*_{2(\kappa^2+2j\kappa+u)}-\Phi^*_{2(\kappa^2+2j\kappa)})\,.
\]
So,
\begin{eqnarray*}
\|C_j\|_{2,\sigma}\lesssim |f_{j,\kappa}|\max_{u} |\alpha_{j,u}|\sum_{0\le u\le \kappa-1} \left(\left\|\sum_{s=2(\kappa^2+2j\kappa)}^{2(\kappa^2+2j\kappa+u)-1}\gamma_s\Phi_s\right\|_{2,\sigma}+\sum_{s=2(\kappa^2+2j\kappa-\kappa)}^{2(\kappa^2+2j\kappa+\kappa)}|\gamma_s|^2\right)+\\
 |f_{j,\kappa}|\max_{u} |\alpha_{j,u}|\sum_{-(\kappa-1)\le u<0} \left(\left\|\sum_{s=2(\kappa^2+2j\kappa+u)}^{2(\kappa^2+2j\kappa)-1}\gamma_s\Phi_s\right\|_{2,\sigma}+\sum_{s=2(\kappa^2+2j\kappa-\kappa)}^{2(\kappa^2+2j\kappa+\kappa)}|\gamma_s|^2\right)\,.
\end{eqnarray*}
Given our assumptions of $f_{j,\kappa},\alpha_{j,u}$ and the
 orthonormality of $\{\phi_s\}$, we get
\[
\|C_j\|_{2,\sigma}\lesssim \left(\sum_{s=2(\kappa^2+2j\kappa-\kappa)}^{2(\kappa^2+2j\kappa+\kappa)}|\gamma_s|^2\right)^{\frac 12}\,.
\]
Hence,
\[
\sum_{j=0}^{\kappa-1}\|C_j\|_{2,\sigma}^2\lesssim \sum_{s>\kappa^2}|\gamma_s|^2\to 0
\]
as $\kappa\to\infty$ and the proof is finished.

\end{proof}

%\noindent {\bf Fill in the details!} Say that the same proof as before works for windows of size $P\kappa$ and then the tails can be estimated individually using the bounds on $\alpha_{j,u}$, i.e., (assuming $d=0$)
%\[
%\sum_{\alpha}\sum_{ |u|\ge P\kappa-1}f_{P\alpha,\kappa}s_{\kappa^2+2P\alpha \kappa+u}\alpha_{P\alpha,u}=\sum_{\beta=0}^{5\kappa^2} s_{\beta}c_\beta
%\]
%where
%\[
%\sum_\beta |c_\beta|^2\le_\ell P^{-2\ell}\,.
%\]

\noindent {\bf Definition.} Take $\ell\in \N$. We will say that a function $h$, defined on $\T$, is $\ell$-localized to an arc $I\subset \T, |I|\sim \kappa^{-1}$ if $|h(z)|\le_\ell (1+P)^{-\ell}$ as long as $\text{dist}_{\T}(z,I)\ge P\kappa^{-1}$ for every $0\le P\lesssim  \kappa$. If $h$ is $\ell$-localized on $I$ for every $\ell\in \N$, we will say that $h$ is $\infty$-localized to $I$.

\begin{lemma}\label{post1}

 Let
\[
\lambda_j(z)=\sum_{u=-(\kappa-1)}^{\kappa-1}\alpha_{j,u}z^u\,,
\] where  $|\lambda_j|$ is $2$-localized to $\{I_j\}$ and the arcs $\{I_j\}|_{j=0}^{\kappa-1}$, which form a disjoint partition of $\T$, satisfy $|I_j|\sim \kappa^{-1}$.  We have for $B$ defined as in \eqref{bdef},
\begin{equation}\label{lisa3}
\|B\|^2_{2,\sigma}-\int\limits_{\T} \sum_{j=0}^{\kappa-1} |f_{j,\kappa}|^2|\lambda_j(z)|^2dm
\to 0, \quad \kappa\to\infty
\end{equation}
for all complex numbers $f_{j,\kappa}$ that satisfy $|f_{j,\kappa}|\lesssim 1$.
\end{lemma}
\begin{proof}
We again write the sum
\begin{equation}\label{sum1}
A:=\sum_{j=0}^{\kappa-1}\sum_{ |u|\le \kappa-1}f_{j,\kappa}s_{\kappa^2+2j\kappa+u}\alpha_{j,u}
\end{equation}
in which all terms are $L^2_\sigma(\T)$-orthogonal. Then,
\[
\|A\|^2_{2,\sigma}=\sum_{j=0}^{\kappa-1} \sum_{|u|\le \kappa-1}|f_{j,\kappa}\alpha_{j,u}|^2=\int \sum_{j=0}^{\kappa-1} |f_{j,\kappa}|^2|\lambda_j(z)|^2dm
\]
by the Plancherel identity. On the other hand, 
\[
A=\sum_{j=0}^{\kappa-1} B_{j}+\sum_{j=0}^{\kappa-1} C_j\,,
\]
where we abbreviate
\[
B_{j}=\sum_{ |u|\le \kappa-1}f_{j,\kappa}\overline{\phi^*_{2(\kappa^2+2j\kappa)}}z^{\kappa^2+2j\kappa+u}\alpha_{j,u}=f_{j,\kappa}\overline{\phi^*_{2(\kappa^2+2j\kappa)}}z^{\kappa^2+2j\kappa}\lambda_j
\]
and
\[
C_j:=\sum_{ |u|\le \kappa-1}f_{j,\kappa}\Bigl(\overline{\phi^*_{2(\kappa^2+2j\kappa+u)}}-\overline{\phi^*_{2(\kappa^2+2j\kappa)}}\Bigr)z^{\kappa^2+2j\kappa+u}\alpha_{j,u}\,.
\]
To finish the proof of the lemma, we only need to show that
\[
\left\|\sum_{j=0}^{\kappa-1} C_j\right\|_{2,\sigma}\to 0, \quad \kappa\to\infty\,.
\]
To this end, we write $A_{m,j}=\{z\in \T:m\kappa^{-1}\leq \dist_{\T}(z,I_j)\leq (m+1)\kappa^{-1}\}$ and decompose $\T=\bigcup_{m}A_{m,j}$, so that 
\begin{align*}
\int\limits_{\T}|\lambda_j|dm\lesssim\int\limits_{A_0}|\lambda_j|dm+\kappa^{-1}\sum_{m\geq 1}(1+m)^{-2}\lesssim \kappa^{-1}.
\end{align*}
%\[
%\int\limits_\T |\lambda_j|dm\lesssim |I_j|\lesssim  \kappa^{-1}\,.
%\]
Then, since $\alpha_{j,u}=\int\limits_{\T}\lambda_jz^{-u}dm$, we have $ |\alpha_{j,u}|\lesssim \kappa^{-1}$, and we may proceed exactly as in Lemma~\ref{lem1}.
\end{proof}

\begin{remark}\label{almostortho} The system
\begin{equation}\label{cot1}
B_{j}=f_{j,\kappa}\overline{\phi^*_{2(\kappa^2+2j\kappa)}}z^{\kappa^2+2j\kappa}\lambda_j
\end{equation}
 is ``almost-orthogonal'' in $L^2_\sigma(\T)$ given our assumption on disjointness of $\{I_j\}$ and $\ell$-localization of $\lambda_j$ with sufficiently large $\ell$. This observation will be used later.
\end{remark}

	On the unit circle, we take a partition of unity $1=\sum_{j=0}^{\kappa-1} \omega_j(z),\, \omega_j(z):=\omega_0(z/z_j)$, $z\in \T$ where we recall $z_j=\exp(2\pi ij/\kappa)$ and $\omega_{0}(e^{ix})=h(\kappa x)$ where $h$ is a suitable smooth nonnegative function with compact support in $[-2\pi,2\pi]$. If
	\begin{equation} \omega_j(z)=\sum_{u\in\Z}\alpha_{j,u}z^u,\,j\in\{0,\dots,\kappa-1\}\,,\label{base_1}
	\end{equation}
	then $\alpha_{j,u}=\alpha_{0,u}z_j^{-u}$ and  $|\alpha_{j,u}|\le_\ell (\kappa (1+(|u|/\kappa)^\ell))^{-1}$ for all $\ell\in \N$.
Clearly, these $\{\omega_j\}$ satisfy conditions of the previous lemma where $I_j$ are arcs of length $4\pi/\kappa$ centered around $z_j$. \medskip

\begin{corollary}\label{cor1} 
Suppose $f_{j,\kappa}$ satisfies $|f_{j,\kappa}|\lesssim  1$.  Consider the partition $\{\omega_j\}|_{j=0}^{\kappa-1}$ and let
\begin{equation}\label{sum909}
A:=\sum_{j=0}^{\kappa-1}\sum_{|u|\le\kappa^2}f_{j,\kappa}s_{\kappa^2+2j\kappa+u}\alpha_{j,u}
\end{equation}
and 
 \begin{align*}
  B=\sum_{j=0}^{\kappa-1}
f_{j,\kappa}\overline{\phi^*_{2(\kappa^2+2j\kappa)}}z^{\kappa^2+2j\kappa}\sum_{u\in\Z}\alpha_{j,u}z^u=\sum_{j=0}^{\kappa-1}
f_{j,\kappa}s_{\kappa^2+2j\kappa}\sum_{u\in\Z}\alpha_{j,u}z^u\,,
 \end{align*}
for $\alpha_{j,u}$ defined as in \eqref{base_1}. Then, we have 
\begin{align}\label{AB0}
\|A-B\|_{2,\sigma}\to 0
\end{align}
and
\begin{equation}\label{pol}
 \sum_{j=0}^{\kappa-1} |f_{j,\kappa}|^2
\|\phi^*_{2(\kappa^2+2j\kappa)}\omega_j\|^2_{2,\sigma} 
-\int\limits_{\mathbb T} \sum_{j=0}^{\kappa-1} |f_{j,\kappa}|^2|\omega_j(z)|^2dm
=o(1), \quad \kappa\to\infty\,.
\end{equation}
\end{corollary}
We provide the proof of the corollary in the Appendix. 

\begin{remark}\label{rr}In all statements of this section, including the next theorem, the indices $n$ of the polynomials $\phi^*_n$ are taken in the form $n=\kappa^2+\kappa j$ or $n=2(\kappa^2+2\kappa j)$. However, our proofs never exploit the arithmetic progression structure in $j$, and we make this choice only for convenience. In fact, all results hold for ensembles $\{\phi^*_{\lambda_{j,\kappa}}\}|_{j=0}^{N}$ where $C_1\kappa^2\le \lambda_{0,\kappa}<\ldots<\lambda_{N,\kappa}\le C_2\kappa^2$ where $N\sim \kappa$ and $\lambda_{j+1,\kappa}-\lambda_{j,\kappa}\sim \kappa$.
\end{remark}

The following quantity's asymptotics will be important in the proof of the main theorem. 
\begin{definition}
Define the quantity  
\begin{align*}
\Sigma_\kappa:=\sum_{j=0}^{\kappa-1}  f_{j,\kappa}\phi^*_{\kappa^2+\kappa j}\omega_j \,,
\end{align*}
where $f_{j,\kappa}$ are complex numbers that satisfy $|f_{j,\kappa}|\lesssim 1$.
\end{definition}
We write $d\sigma=\sigma'dm+d\sigma_s$. Under the Szeg\H o condition (see \cite[Theorem 2.4.1, p.~144]{SimonOPUC1}), $\phi^*_n\to \Pi$ in $L^2_\sigma(\T)$ where $\Pi:=D^{-1}\cdot \mathbb{1}_{E_{{\rm ac},\sigma}}$, $D$ is the Szeg\H{o} function \eqref{eq:Szegofxn}, and the set $E_{{\rm ac},\sigma}$ is the complement of the support of the singular part of $\sigma$. Our immediate goal will be to show that
\begin{equation}\label{main}
\left\|\Sigma_\kappa-\Pi\sum_{j=0}^{\kappa-1} f_{j,\kappa}\omega_j\right\|_{2,\sigma}\to 0,\,\kappa\to \infty
\end{equation}
under various assumptions. This serves as one of the main inputs in the proof of Theorem~\ref{thm:main}, and we begin by proving it is implied by \eqref{loga}. \medskip

\begin{theorem}\label{lem:PCA} Suppose  $|f_{j,\kappa}|\lesssim 1$.  Then, \eqref{loga} implies \eqref{main}.
%\begin{equation}\label{main7}
%\Sigma_\kappa-\Pi\sum_jf_{j,\kappa}\omega_j\to 0, \quad\kappa\to\infty
%\end{equation}
%in $L^2_\sigma$. 

\end{theorem}
\begin{proof} Assume without loss of generality that $\kappa$ is even and let $\kappa'=\kappa/2-1$. It is enough to show that
\[
\left\|\sum_{j\, {\rm even}} f_{j,\kappa}(\phi^*_{\kappa^2+\kappa j}-\Pi)\omega_j\right\|_{2,\sigma} \to 0, \quad \left\|\sum_{j\, {\rm odd}} f_{j,\kappa}(\phi^*_{\kappa^2+\kappa j}-\Pi)\omega_j\right\|_{2,\sigma} \to 0
\]
as $\kappa\to\infty$. We will prove the first bound, the second one is similar. Since $\omega_{j_1}\omega_{j_2}=0$ if $\dist_\T(z_{j_1},z_{j_2})>2\pi/\kappa$, the terms in the sum are orthogonal and 
\[
\left\|\sum_{j\, {\rm even}} f_{j,\kappa}(\phi^*_{\kappa^2+\kappa j}-\Pi)\omega_j\right\|_{2,\sigma}^2=\sum_{j=0}^{\kappa'}|f_{2j,\kappa}|^2\int |\phi^*_{\kappa^2+2\kappa j}-\Pi|^2\omega^2_{2j}d\sigma\lesssim \sum_{j=0}^{\kappa'}\int |\phi^*_{\kappa^2+2\kappa j}-\Pi|^2\omega^2_{2j}d\sigma
\]
and it is enough to assume that $f_{j,\kappa}=1$ for all $j$.
Since $\phi_{n}^*\to \Pi$ in $L^2_\sigma(\T)$, we only need to  prove
\[
\sum_{j=0}^{\kappa'}\int |\phi^*_{\kappa^2+2\kappa j}-\phi^*_{2\kappa^2}|^2\omega^2_{2j}d\sigma\to 0\,.
\] In the arguments that follow, we will often make no significant distinction between quantities \[\sum_{j=0}^{\kappa'} \phi^*_{\kappa^2+2j\kappa} \omega_{2j} \quad \text{ and }\quad \sum_{j=0}^{\kappa'} \Phi^*_{\kappa^2+2j\kappa} \omega_{2j}\] for large $\kappa$. In fact, they are asymptotically proportional to each other. This is due to the inequality
\begin{eqnarray}\label{observe-1}
\left\|\sum_{j=0}^{\kappa'} \phi^*_{\kappa^2+2j\kappa}\omega_{2j}-\nu_{\infty}\sum_{j=0}^{\kappa'}\Phi^*_{\kappa^2+2j\kappa} \omega_{2j}\right\|_{2,\sigma}^2=
\hspace{6cm}\\
\sum_{j=0}^{\kappa'}\|(\nu_{\kappa^2+2j\kappa}\Phi^*_{\kappa^2+2j\kappa}-\nu_{\infty}\Phi^*_{\kappa^2+2j\kappa}) \omega_{2j}\|^2_{2,\sigma}\stackrel{\eqref{opuc2}}{=} o(1) \sum_{j=0}^{\kappa'}\|\phi^*_{\kappa^2+2j\kappa} \omega_j\|^2_{2,\sigma}\stackrel{\eqref{pol}}{=}o(1),\,\nonumber
\end{eqnarray}
when $\kappa\to\infty$. Given that observation, \eqref{opuc1} and \eqref{opuc2}, we only need to prove that 
\begin{equation}\label{omeg-n}
D_\kappa:=\sum_{j=0}^{\kappa'} \left(\sum_{s=\kappa^2+2\kappa j}^{2\kappa^2-1} \gamma'_s \phi_s\right)\omega_{2j}=o(1), \quad \kappa\to\infty
\end{equation}
in $L^2_\sigma(\T)$, where
\[
\gamma'_\alpha:=\gamma_\alpha\cdot \prod_{\beta=\kappa^2}^{\alpha-1}(1-|\gamma_\beta|^2)^{\frac 12}, \quad \alpha\in [\kappa^2,2\kappa^2]\,.
\]
Notice that the quantity $\|D_\kappa\|_{2,\sigma}^2$ involves the sum of terms $\langle \phi_\alpha \omega_{2j},\phi_\beta\omega_{2j}\rangle_\sigma$, where $\alpha,\beta\ge \kappa^2$.  
If $\omega^2_0(z)=:\sum_{u\in \Z}\alpha'_{u,\kappa}z^u$,
then $\omega^2_j(z)=\sum_{u\in \Z}\alpha'_{j,u,\kappa}z^u$ where
$\alpha'_{j,u,\kappa}=\alpha'_{u,\kappa}z_j^{-u}$. The functions $\{\omega^2_j\}$ are real-valued and 
\begin{equation}\label{prod-fur}
|\alpha'_{j,u,\kappa}|\le_\ell (\kappa (1+(|u|/\kappa)^\ell))^{-1}
\end{equation}
for every $\ell\in \N$.
 Therefore,
for $\alpha,\beta\in [\kappa^2,2\kappa^2]$ and any fixed $\delta\in (0,1)$, one can write
\begin{equation}\label{opuc3}
\langle \phi_\alpha \omega_{2j},\phi_\beta\omega_{2j}\rangle_\sigma=:\sum_{|u|<\kappa^{1+\delta}}\bar{\alpha}_{2j,u,\kappa}'\langle \phi_\alpha,z^u\phi_\beta\rangle_\sigma+R\,,
\end{equation}
where, taking $\ell\ge 2,$ we have a bound for the tail $R$:
\begin{align*}
|R|= \left|\sum_{|u|\geq \kappa^{1+\delta}}\bar{\alpha}_{2j,u,\kappa}'\langle \phi_\alpha,z^u\phi_\beta\rangle_\sigma\right|\leq \sum_{|u|\geq \kappa^{1+\delta}}|{\alpha}_{2j,u,\kappa}'|\le_{\ell} \sum_{|u|\geq \kappa^{1+\delta}}(\kappa (1+(|u|/\kappa)^\ell))^{-1}\le_\ell \kappa^{-\delta(\ell-1)}\to 0
\end{align*}
by the Cauchy-Schwarz inequality using $\|\phi_\alpha\|_{2,\sigma}=\|z^{u}\phi_{\beta}\|_{2,\sigma}=1$. Since $\ell$ can be taken arbitrarily large, the contribution from $R$ is negligible.\medskip

Now, focus on the first term in \eqref{opuc3}. Since
$
\langle \phi_\alpha,z^u\phi_\beta\rangle_\sigma=\overline{\langle \phi_\beta,z^{-u}\phi_\alpha\rangle_\sigma}
$, 
it will be sufficient to consider $u\ge 0$ only. Recall that $\rho_j=(1-|\gamma_j|^2)^{\frac 12}$ and $\gamma_{-1}=-1$. 
We use the following identity  (see \cite{am}, formula (2.25)):
\begin{equation}\label{form-g}
G_{k,l}:=\langle z\phi_\ell,\phi_k\rangle_\sigma=\left\{
\begin{array}{cc}
-\bar\gamma_\ell \gamma_{k-1}\prod_{j=k}^{\ell-1}\rho_j, & 0\le k\le \ell,\\
\rho_\ell,& k=\ell+1,\\
0, & k\ge \ell+2
\end{array}
\right.
\end{equation}
and we changed the order in the inner product writing $\langle z\phi_\ell,\phi_k\rangle_\sigma$ instead of $\langle \phi_k,z\phi_\ell\rangle_\sigma$ since \cite{am} assumes that $\langle v_1,v_2\rangle$ is linear in the second vector $v_2$.
Define $H$ as the closure of all polynomials in $L^2_\sigma(\T)$, it is a Hilbert space and a subspace in $L^2_\sigma(\T)$. The system $\{\phi_j\}|_{j=0}^\infty$ forms an orthonormal basis in $H$ and 
the matrix $G$ represents the multiplication operator $f\mapsto zf$, an isometry in $H$, in this basis. Notice that for $k,\ell\ge \kappa^2$, the formula \eqref{form-g} involves parameters $\{\gamma_s\}$ for $s\ge \kappa^2-1$ only. Such a form of the matrix $G$ implies that, e.g., for the elements of the square $G^2$, denoted by $G^{(2)}_{i,j}$, we have
\[
G^{(2)}_{i,j}=\sum_{i-1\le d\le j+1}G_{i,d}G_{d,j}
\]
and $G^{(2)}_{i,j}=0$ if $i\ge j+3$. Arguing by induction,  one  shows that $G^{(s)}_{i,j}$ depends only on $\{\gamma_k\}$ with $k\ge \kappa^2-s-1$ as long as $i,j\ge \kappa^2$. Alternatively, one can iterate the formula (2.29) from \cite{am} $s-1$ times to come to the same conclusion.

Then, since $\alpha,\beta\le 2\kappa^2$ and $s\le \kappa^{1+\delta}$, we see that $\|D_\kappa\|_{2,\sigma}^2$  can be written explicitly as a function of the Verblunsky parameters $\{\gamma_j\}$ with $\kappa^2/2\le j\le 3\kappa^2$ whenever $\kappa$ is sufficiently large. We define the new system of recursion parameters $\{\widetilde\gamma_j\}$ as follows: $\widetilde \gamma_j=0$ if $ j\le \kappa^2/2$ or $j\ge 3\kappa^2$ and $\widetilde \gamma_j=\gamma_j$ for the other $j$.  Denote the corresponding measure and polynomials by $\widetilde \sigma$ and  $\widetilde \phi_j$, respectively. Notice that they all depend on $\kappa$.  Recall that our goal is to prove \eqref{omeg-n}, in which we can now replace $\sigma$ and $\phi_s$ by $\widetilde\sigma$ and $\widetilde \phi_s$. To this end, we apply \eqref{pol} and Remark \ref{rr} to this new system. That gives
\[
\int\limits_{\mathbb T} \sum_{j=0}^{\kappa'} |\widetilde\phi^*_{\kappa^2+2j\kappa}(z)|^2 \omega^2_{2j}(z)d\widetilde\sigma- \int\limits_{\mathbb T} \sum_{j=0}^{\kappa'}  \omega^2_{2j}(z)dm=o(1), \quad \kappa\to\infty
\]
and, by the observation \eqref{observe-1} made above,
\[
\int\limits_{\mathbb T} \sum_{j=0}^{\kappa'} |\widetilde\Phi^*_{\kappa^2+2j\kappa}(z)|^2 \omega^2_{2j}(z)d\widetilde\sigma- \int\limits_{\mathbb T} \sum_{j=0}^{\kappa'}  \omega^2_{2j}(z)dm=o(1), \quad \kappa\to\infty\,.
\]
By the OPUC recursion, the last bound is equivalent to
\begin{equation}\label{raz}
\int\limits_{\mathbb T} \sum_{j=0}^{\kappa'}  |1-z\sum_{s=\kappa^2/2}^{\kappa^2+2j\kappa-1}\widetilde\gamma_s\widetilde\Phi_{s}(z)|^2\omega^2_{2j}(z)d\widetilde\sigma- \int\limits_{\mathbb T} \sum_{j=0}^{\kappa'}\omega^2_{2j}(z)dm=o(1), \quad \kappa\to\infty.
\end{equation}
We write for the first term
\begin{equation}\label{sum-t}
\int\limits_{\mathbb T} \sum_{j=0}^{\kappa'}  |1-z\sum_{s=\kappa^2/2}^{\kappa^2+2j\kappa-1}\widetilde\gamma_s\widetilde\Phi_{s}(z)|^2  \omega^2_{2j}(z)d\widetilde\sigma=I_1+I_2+I_3\,,
\end{equation}
where
\begin{equation}\label{dva}
I_1=\int\limits_{\mathbb T} \sum_{j=0}^{\kappa'} \omega^2_{2j}(z)d\widetilde\sigma=\int\limits_{\mathbb T} \sum_{j=0}^{\kappa'}  \omega^2_{2j}(z)dm+o(1)\,.
\end{equation}
Here, we replaced $d\widetilde\sigma$ by $dm$ because the first $\kappa^2/2$ moments for $\widetilde \sigma$ are the same as for $dm$ (since $\widetilde\gamma_j=0, j\le \kappa^2/2$, see, e.g., \cite{SimonOPUC1}) and  the Fourier coefficients of the function ${\omega_{2j}^2}$ of index $n$ are negligible for $|n|>\kappa^{1+\delta}$ by \eqref{prod-fur}. Then, we consider
\begin{equation}\label{sum4}
I_2=-2\Re \sum_{j=0}^{\kappa'} \int\limits_{\mathbb T} z\sum_{s=\kappa^2/2}^{\kappa^2+2j\kappa-1}\widetilde\gamma_s\widetilde\Phi_{s}\omega_{2j}^2d\widetilde \sigma\,.
\end{equation}
Define $\Gamma(s,u):=\langle z^{u+1}\widetilde\phi_s(z),1\rangle_{\widetilde\sigma}$. By Plancherel's theorem, we have
\begin{equation}\label{gam-oh}
\sum_{s\ge 0} |\Gamma(s,u)|^2\le 1\,.
\end{equation}
Recalling that $\widetilde \gamma_j=0$ for $j\le \kappa^2/2$ and $j\ge 3\kappa^2$, we get 
\begin{align}\nonumber
\tilde \gamma_s':=\tilde \gamma_s\prod_{\ell\leq s-1}   (1-|\tilde \gamma_\ell|^2)^{1/2},\hspace{8cm}\\
\label{lota}
\sum_{j=0}^{\kappa'} \int z\sum_{s=\kappa^2/2}^{\kappa^2+2j\kappa-1}\widetilde\gamma_s\widetilde\Phi_{s}\omega_{2j}^2d\widetilde \sigma=\sum_{|u|\le 0.1\kappa^2}\alpha'_{u,\kappa}\sum_{j=0}^{\kappa'}\left(\sum_{s=\kappa^2/2}^{\kappa^2+2j\kappa-1} \widetilde\gamma_s'\Gamma(s,u)\right)e^{-4\pi iju/\kappa}+o(1)\,,\
\end{align}
where $o(1)$ captures the range $|u|> 0.1\kappa^2$ which gives a negligible contribution due to the strong decay of $\alpha'_{u,\kappa}$. Consider the main term
\begin{align*}
\sum_{|u|\le 0.1\kappa^2}\alpha'_{u,\kappa}\sum_{j=0}^{\kappa'}\left(\sum_{s=\kappa^2/2}^{\kappa^2+2j\kappa-1} \widetilde\gamma_s'\Gamma(s,u)\right)e^{-4\pi iju/\kappa}\,.
\end{align*}
For $u \equiv 0 \mod \kappa'+1 $, we use $|\alpha'_{n(\kappa'+1),\kappa}|\le_{\ell} \kappa^{-1}(|n|+1)^{-\ell}$ and the Cauchy-Schwarz inequality to write
\[
\sum_n\left|\alpha'_{n(\kappa'+1),\kappa}\sum_{j=0}^{\kappa'}\left(\sum_{s=\kappa^2/2}^{\kappa^2+2j\kappa-1} \widetilde\gamma_s'\Gamma(s,n(\kappa'+1))\right)\right|\stackrel{\eqref{gam-oh}}{\le_{\ell}} \|\widetilde\gamma\|_2\sum_n(|n|+1)^{-\ell}=o(1)
\]
as $\kappa\to\infty$ when $\ell\ge 2$ because $\|\widetilde\gamma\|_2\to 0$ as $\kappa\to\infty$.
For $u\not\equiv 0 \mod \kappa'+1$, 
we sum by parts in $j$ using the formula
\[
\sum_{k=0}^n h_kg_k=h_n\sum_{k=0}^ng_k-\sum_{j=0}^{n-1}(h_{j+1}-h_j)\sum_{k=0}^j g_k
\]
and
\[
\sum_{j=0}^{\kappa'}e^{-4\pi iju/\kappa}=0, \quad u\not\equiv 0\mod\kappa'+1
\]
to get
\begin{eqnarray*}
\sum_{j=0}^{\kappa'}\left(\sum_{s=\kappa^2/2}^{\kappa^2+2j\kappa-1} \widetilde\gamma_s'\Gamma(s,u)\right)e^{-4\pi iju/\kappa}=
\\-\sum_{j=0}^{\kappa'-1}\left(\sum_{s=\kappa^2+2j\kappa}^{\kappa^2+2(j+1)\kappa-1} \widetilde\gamma_s'\Gamma(s,u)\right)\frac{1-e^{-4\pi i(j+1)u/\kappa}}{1-e^{-4\pi iu/\kappa}}\,.
\end{eqnarray*}
Notice that 
\[
\left|\frac{1-e^{-4\pi i(j+1)p/\kappa}}{1-e^{-4\pi ip/\kappa}}\right|\lesssim \frac{\kappa}{p}, \quad 1\le p\le \kappa'\,.
\]
We write $u=n(\kappa'+1)+p$, where $p\in \{1,\ldots,\kappa'\}, n\in \Z$ and get, after applying the Cauchy-Schwarz inequality to the sum in $s$ and then to the sum in $j$:
\begin{eqnarray}\nonumber
|I_2|\le_{\ell} o(1)+\hspace{12cm}\\\nonumber
+\kappa^{-1}\sum_{n} (1+|n|)^{-\ell}\sum_{p=1}^{\kappa'}\frac{\kappa}{p}
\sum_{j=0}^{\kappa'-1}\left(\sum_{s=\kappa^2+2j\kappa}^{\kappa^2+2(j+1)\kappa-1}|\gamma_s|^2\right)^{\frac 12}\left(\sum_{s=\kappa^2+2j\kappa}^{\kappa^2+2(j+1)\kappa-1}
|\Gamma(s,n(\kappa'+1)+p)|^2\right)^{\frac 12}\le_{\ell}\\
 o(1)+\kappa^{-1}\sum_{n} (1+|n|)^{-\ell}\sum_{p=1}^{\kappa'}\frac{\kappa}{p}\left(\sum_{j=0}^{\kappa'-1}\sum_{s=\kappa^2+2j\kappa}^{\kappa^2+2(j+1)\kappa-1}|\gamma_s|^2\right)^{\frac 12}\left(\sum_{j=0}^{\kappa'-1}\sum_{s=\kappa^2+2j\kappa}^{\kappa^2+2(j+1)\kappa-1}
|\Gamma(s,n(\kappa'+1)+p)|^2\right)^{\frac 12}\,.
\nonumber
\end{eqnarray}
For any fixed $p$, we have
\[
\sum_{j=0}^{\kappa'-1}\sum_{s=\kappa^2+2j\kappa}^{\kappa^2+2(j+1)\kappa-1}|\gamma_s|^2
\le\sum_{s=\kappa^2}^{2\kappa^2}|\gamma_s|^2, \quad
\sum_{j=0}^{\kappa'-1}\sum_{s=\kappa^2+2j\kappa}^{\kappa^2+2(j+1)\kappa-1}
|\Gamma(s,n(\kappa'+1)+p)|^2\le\sum_{s=\kappa^2}^{2\kappa^2}|\Gamma(s,n(\kappa'+1)+p)|^2\,.
\]
We want to estimate the last quantity
\[
\sum_{s=\kappa^2}^{2\kappa^2}|\Gamma(s,u)|^2, \quad u=n(\kappa'+1)+p\,.
\]
Recall that the index $u$ satisfies $|u|\le 0.1\kappa^{2}$. Since $s\ge \kappa^2$, we have $\Gamma(s,u)=\langle z^{u+1}\widetilde \phi_s,1\rangle_{\widetilde\sigma}=0$ for negative $u$ in that range by orthogonality. Hence, we only need to consider $u\ge 0$. Let $\widetilde H$ denote the subspace obtained as the closure of all polynomials in $L^2_{\widetilde\sigma}(\T)$. For any $v\in L^2_{\widetilde\sigma}(\T)$, we have
\begin{equation}\label{pyth}
\|v\|_{2,\widetilde\sigma}^2=\|{\rm Proj}_{\widetilde H}v\|_{2,\widetilde\sigma}^2+\|{\rm Perp}_{\widetilde H}v\|_{2,\widetilde\sigma}^2
\end{equation}
by the Pythagorean Theorem.
Notice that $\|{\rm Perp}_{\widetilde H}v\|_{2,\widetilde\sigma}={\rm dist}_{2,\widetilde\sigma}(v,\widetilde H)$ and one gets ${\rm dist}_{2,\widetilde\sigma}(z^{-1},\widetilde H)\le {\rm dist}_{2,\widetilde\sigma}(z^{-1},z^u\widetilde H)={\rm dist}_{2,\widetilde\sigma}(z^{-u-1},\widetilde H)$ for any $u\ge 1$. Since $\|z^{-u}\|_{2,\widetilde\sigma}=1$, the identity \eqref{pyth} gives $\|{\rm Proj}_{\widetilde H}(z^{-u})\|_{2,\widetilde\sigma}\le \|{\rm Proj}_{\widetilde H}(z^{-1})\|_{2,\widetilde\sigma}$. Notice that
\[
\|{\rm Proj}_{\widetilde H}(z^{-u-1})\|_{2,\widetilde\sigma}^2=\sum_{s\ge 0}|\Gamma(s,u)|^2
\]
and $\|{\rm Proj}_{\widetilde H}(z^{-1})\|_{2,\widetilde\sigma}^2=1-{\rm dist}^2_{2,\widetilde\sigma}(z^{-1},\widetilde H)=1-\prod_{j\ge 0}(1-|\widetilde\gamma_j|^2)\lesssim \|\widetilde\gamma\|_2^2$ by the Szeg\H{o} theorem. Hence, 
\[
\sum_{s=\kappa^2}^{2\kappa^2}|\Gamma(s,n(\kappa'+1)+p)|^2\lesssim  \|\widetilde\gamma\|_2^2
\]
and
\begin{equation}\label{lot1}
|I_2|\lesssim  \log\kappa\left(\sum_{s=\kappa^2/2}^{3\kappa^2}
|\gamma_s|^2\right)+o(1).
\end{equation}
So that $I_2\to 0$ as $\kappa\to \infty$ since \eqref{loga} holds. \medskip
Recall \eqref{raz}, \eqref{dva} and that 
 $I_2\to 0$. Then, for the last term in \eqref{sum-t}, we get (compare with \eqref{omeg-n}, which we need to show for the $\widetilde\sigma$-system)
\[
I_3:=\int\limits_{\mathbb T}\sum_{j=0}^{\kappa'}  \left|\sum_{s=\kappa^2/2}^{\kappa^2+2j\kappa-1}\widetilde\gamma_s\widetilde\Phi_{s}(z)\right|^2  \omega^2_{2j}(z)d\widetilde\sigma=o(1), \quad \kappa\to\infty\,,
\]
which is the same as 
\begin{equation}\label{tri}
\int\limits_{\mathbb T} \sum_{j=0}^{\kappa'} |\widetilde\Phi^*_{\kappa^2+2j\kappa}-\widetilde \Phi^*_{\kappa^2/2}|^2  \omega^2_{2j}(z)d\widetilde\sigma =o(1), \quad \kappa\to\infty\,.
\end{equation}
We claim that
\begin{eqnarray}\label{chet}
\int\limits_{\T}  |\widetilde\Phi^*_{2\kappa^2}-\widetilde \Phi^*_{\kappa^2/2}|^2 d\widetilde\sigma\to 0, \quad \kappa\to\infty.
\end{eqnarray}
Indeed, by the Bernstein-Szeg\H{o} approximation \cite[Theorem~1.7.8]{SimonOPUC1}, it suffices to show
\begin{align*}
\int\limits_{\mathbb T}  |\widetilde\phi^*_{2\kappa^2}-\widetilde \phi^*_{\kappa^2/2}|^2 \frac{dm}{|\widetilde\phi^*_{2\kappa^2}|^2}\to 0, \quad \kappa\to\infty\,.
\end{align*}
To this end, we may compute 
\begin{eqnarray*}
\int\limits_{\mathbb T}  |\widetilde\phi^*_{2\kappa^2}-\widetilde \phi^*_{\kappa^2/2}|^2 \frac{dm}{|\widetilde\phi^*_{2\kappa^2}|^2}
&=2\left(1-\int\limits_{\mathbb T} \Re  \frac{\widetilde\phi^*_{\kappa^2/2}}{\widetilde\phi^*_{2\kappa^2}}dm\right)=2\left(1-\Re  \frac{\widetilde\phi^*_{\kappa^2/2}(0)}{\widetilde\phi^*_{2\kappa^2}(0)}\right)\\
&=2\left(1-\prod_{j=\kappa^2/2}^{2\kappa^2-1}(1-|\widetilde \gamma_j|^2)^{1/2}\right)\to 0
\end{eqnarray*}
as $\kappa\to \infty$. Putting \eqref{tri} and \eqref{chet} together, we have
\[
\int \sum_{j=0}^{\kappa'} |\widetilde\Phi^*_{\kappa^2+2j\kappa}-\widetilde \Phi^*_{2\kappa^2}|^2 \omega^2_{2j}(z)d\widetilde\sigma=o(1), \quad \kappa\to\infty\,,
\]
which is equivalent to \eqref{omeg-n} for the $\widetilde\sigma$ system. As we explained already, \eqref{omeg-n} for the $\widetilde\sigma$ system implies \eqref{omeg-n} for $\sigma$-system, and that finishes the proof. \end{proof}
\medskip

\section{Existence and Completeness of Wave Operators}\label{mainsection}

In this section, we modify the partition of unity defined above to depend on a continuous parameter. Throughout this section, $\kappa$ will be the integer related to $T$ by $\kappa:=\lfloor \sqrt{T} \rfloor$, and we form the partition as before, with
\begin{align*}
\omega_{(T)}(e^{ix})=h(\kappa x),
\end{align*}
where $\omega_{j}(z)=\omega_{(T)}(z/z_j)$, and with roots of unity $z_j=\exp(2\pi i j/\kappa )$ and $h\geq 0$ again chosen so that 
\begin{align*}
1=\sum_{j=0}^{\kappa-1}\omega_j(z).
\end{align*}

Before proving Theorem~\ref{thm:main}, we first prove a lemma allowing for the control of the evolution of the bump functions $\omega_j(z)$ just defined above. It is a direct application of the method of nonstationary phase.
\begin{lemma}\label{lem:evobump}
Let $T>0$ and 
\begin{align}\label{defEJ}
E_j(n)&:=\langle e^{iT\cos(x)}\omega_j(z),z^{n}\rangle_{m}=\frac1{2\pi}\int\limits_{0}^{2\pi} e^{i(T \cos x-nx)}\omega_j(e^{ix})dx
\end{align}
and define $n_j=-T\sin(2\pi j/\kappa)$. Then, for any $\ell\in \Z_+$, we have 
\begin{align}\label{eq:rapid}
|E_j(\lceil n_j\rceil+u)|\leq_{\ell}\kappa^{-1}\left(|u|/\kappa+1 \right)^{-\ell}.
\end{align}

\end{lemma}
\begin{proof}
With $x_j=2\pi j/\kappa$, we write
\begin{align*}
E_j(n)&=\frac{1}{2\pi}\int\limits_{x_j-\frac{2\pi}{\kappa}}^{x_j+\frac{2\pi}{\kappa}}e^{i(T \cos x-nx)}\omega_j(e^{ix})dx\,, 
\end{align*}
since $\supp \,\omega_j\subset\{ e^{ix}:x\in[x_j-2\pi/\kappa,x_j+2\pi/\kappa] \}$.
Changing variables $x\mapsto x_j+s/\kappa$, we have
\begin{align*}
    E_j(n)=\frac{1}{2\pi\kappa}\int\limits_{-2\pi}^{2\pi }e^{i(T\cos(x_j+\kappa^{-1} s)-n(\kappa^{-1}s+x_j))}h(s)ds\,,
\end{align*}
since $\omega_j(e^{i(x_j+\kappa^{-1} s)})=h(s)$. Taylor expanding the phase around $s=0$, we find
\begin{align*}
    E_j(n)&=\frac{1}{2\pi\kappa}\int\limits_{-2\pi}^{2\pi }\exp\left(-is (n+T \sin(x_j))/\kappa\right)g_T(s)ds\,,
\end{align*}
where $g_T(s)$ is compactly supported in the interval $[-2\pi,2\pi]$, with $g_T$ and all of its derivatives uniformly bounded in $T$ and $j$.
Notice that $n_j=-T\sin(x_j)$ is the point where the phase of the oscillatory part is stationary.  We may introduce $u=n+T\sin x_j$ and apply the method of nonstationary phase \cite{muscalu-schlag} to the integral
\begin{align*}
    \frac{1}{2\pi\kappa}\int\limits_{-2\pi}^{2\pi }\exp(-is(u/\kappa))g_T(s)ds
\end{align*}
to conclude \eqref{eq:rapid}.
\end{proof}

%\begin{lemma}\label{lem:Hardy}
%Let $T>0$. Suppose $\supp(\omega_j)\subseteq \T\setminus \{e^{i\theta}, \theta\in [-\delta,\delta]\cup [\pi-\delta,\pi+\delta]\}$ and let $E_j$ be as in \eqref{defEJ}. Then, for any $\ell\in \N$, 
%\begin{align}\label{eq:pointwisebound}
%\sup_{z\in \mathbb T}\left|e^{-iT\cos(x)}\sum_{s\geq 1}E_j(s)z^s-\mathbb{1}_{\{\Im(z)<0\}}(z)\omega_j(z)  \right|\le_{\ell,\delta} T^{-\ell}.
%\end{align}
%\end{lemma}
%\begin{proof}
%With the support condition, and taking $T$ large enough, we may assume $\omega_j$ is supported on either the upper or lower half of the circle. If $\omega_j$ is supported in $\{z:\Im(z)>0\}$, we  may apply the previous lemma to find for any $\ell\geq 2$, 
%\begin{align*}
%\left| e^{-iT\cos(x)}\sum_{s\geq 1}E_j(s)z^s\right|\lesssim \sum_{s\geq 1}\left|\int\limits_{0}^{2\pi}e^{i(T\cos(x)-sx)}\omega_j(e^{ix})dx\right| \le_{\ell,\delta}  \sum_{s\ge 0} \frac{1}{\kappa(1+|s-n_j|/\kappa)^\ell}\le_{\ell,\delta} \kappa^{-\ell}\,.
%\end{align*}
%Meanwhile, if $\omega_j(z)$ is supported in $ \{z:\Im(z)<0\}$, we write by Fourier inversion
%\begin{align*}
%e^{-iT\cos(x)}\sum_{s\geq 1}E_j(s)z^s=\omega_j(z)-e^{-iT\cos(x)}\sum_{s\leq 0}E_j(s)z^s
%\end{align*}
%and we may again estimate the last term similarly.
%\end{proof}

We may now prove our main theorem. The previous lemma allows for the control of the free evolution. Instead, to control the $J$-evolution, we make use of the results in the previous section.

\begin{proof}[Proof of Theorem~\ref{thm:main}.]

First, we compute the asymptotics of $e^{iTJ_0}f$ for $f$ in a dense subset of $\ell^2(\N)$. To this end, we extend $f\in \ell^2(\N)$ to $f^{ (\rm odd)}=\{f^{ (\rm odd)}_n\}\in\ell^2(\Z)$ in the odd fashion with $f^{ (\rm odd)}_0=0$, and define $J_0^{\mathbb Z}$ on $\Z$ as the discrete free Schr\"odinger operator. Then, $\{ e^{iTJ_0}f\}_{n\in \N}$ is the restriction of $\{e^{iTJ_0^{\Z}}f^{\rm (odd)}\}_{n\in \Z}$ to $\ell^2(\N)$. We compute the evolution $e^{iTJ_0^\Z}$ by passing to the Fourier side, and examine
$
e^{iT\cos x}\widehat f\,,
$
for $\widehat f$ an odd $L^2_m(\T)$ function. We denote this subspace by
\begin{align*}
L^2_{m,\rm odd}(\T):=\{ g\in L^2_m(\T): g(z)=-g(z^{-1})\}.
\end{align*}
Take $\widehat f\in L^2_{m,\rm odd}(\T)\cap C^\infty(\T)$ and suppose it satisfies
\begin{align}\label{momentumsupport}
\supp \widehat f\cap \{e^{i\theta}:\theta\in[0,\pi]\}
\subset
\{e^{i\theta}:\theta\in [\delta,\pi/2-\delta]\cup [\pi/2+\delta,\pi-\delta]\}
\end{align}
with some positive $\widehat f$-dependent parameter $\delta$. The set of $\widehat{f}$ satisfying these conditions is dense in $L^2_{m,\rm odd}(\T)$. By scaling, we may also assume that $\|\widehat f\|_{L^\infty(\T)}\le 1$.

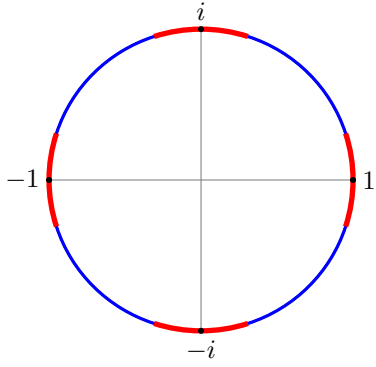
\begin{center}
\begin{tikzpicture}[scale=2]
    \def\r{1}
    \def\d{18} % visual delta in degrees

    % Circle
    \draw[thick] (0,0) circle (\r);

    % Removed delta-neighborhoods
    \draw[line width=2pt, red] (-\d:\r) arc[start angle=-\d, end angle=\d, radius=\r];
    \draw[line width=2pt, red] (90-\d:\r) arc[start angle=90-\d, end angle=90+\d, radius=\r];
    \draw[line width=2pt, red] (180-\d:\r) arc[start angle=180-\d, end angle=180+\d, radius=\r];
    \draw[line width=2pt, red] (270-\d:\r) arc[start angle=270-\d, end angle=270+\d, radius=\r];

    % Allowed arcs
    \draw[line width=1.2pt, blue] (\d:\r) arc[start angle=\d, end angle=90-\d, radius=\r];
    \draw[line width=1.2pt, blue] (90+\d:\r) arc[start angle=90+\d, end angle=180-\d, radius=\r];
    \draw[line width=1.2pt, blue] (180+\d:\r) arc[start angle=180+\d, end angle=270-\d, radius=\r];
    \draw[line width=1.2pt, blue] (270+\d:\r) arc[start angle=270+\d, end angle=360-\d, radius=\r];

    % Axes ending at the circle
    \draw[gray] (-\r,0) -- (\r,0);
    \draw[gray] (0,-\r) -- (0,\r);

    % Mark special points
    \fill (1,0) circle (0.02);
    \fill (0,1) circle (0.02);
    \fill (-1,0) circle (0.02);
    \fill (0,-1) circle (0.02);

    % Labels
    \node[right] at (1,0) {$1$};
    \node[above] at (0,1) {$i$};
    \node[left] at (-1,0) {$-1$};
    \node[below] at (0,-1) {$-i$};

    % Description
    \node at (0,-1.55) {The set $\supp(\hat{f})$ is depicted in blue, with the excised arcs in red.};

\end{tikzpicture}
\end{center}

To obtain the asymptotics of $e^{iTJ_0^{\Z}}$ in the physical space $\ell^2(\Z)$, we wavepacket decompose $\hat{f}$. Consider the partition $\{\omega_j\}$ as in the preamble and write $\widehat f=\sum_{j=0}^{\kappa-1} \widehat f\omega_j$. By smoothness of $\widehat f$, we conclude that 
\[
\|\widehat f(z)-\sum_{j=0}^{\kappa-1} \widehat f(z_j) \omega_j\|_{2}\to 0, \quad T\to\infty\,.
\]
%Also, $|\widehat f(z_{j+1})-\widehat f(z_{j})|\lesssim T^{-1/2}$.
Since the evolution preserves the $L^2_m(\T)$-norm, we only need to compute the evolution of an individual bump which may be controlled through Lemma~\ref{lem:evobump}; setting for $E_j(n)$ as in \eqref{defEJ}, by Fourier inversion, we have 
\begin{align}\label{eq:fourierapprox}
\sum_{j=0}^{\kappa-1} \widehat f(z_j)\sum_{n\in\mathbb Z} E_j(n)z^{n}-e^{iT\cos x}\widehat f(z)\to 0, \quad T\to\infty
\end{align}
in $L^2_m(\T)$. Thus, using that $e^{iTJ_0}f$ is the restriction of the full-line evolution $e^{itJ_0^\Z}$ to $\ell^2(\N)$,
we examine the quantity 
\begin{align}
e^{-iTJ}\left\{ \sum_{j=0}^{\kappa-1}\widehat{f}(z_j)E_j(n)\right\}_{n\geq 1}.
\end{align}
After applying the spectral transformation for $J$, $U:\ell^2(\N)\to L^2_{\rho}([-1,1])$, given by 
\begin{align}\label{spectralmap}
( Uf)(\lambda)=\sum_{n=1}^\infty f_np_{n-1}(\lambda)
 \end{align}
 for $f=\{f_n\}\in \ell^2(\N)$, we may instead examine
\begin{align*}
e^{-iT\lambda}\sum_{n=1}^\infty \left(\sum_{j=0}^{\kappa-1}\widehat{f}(z_j)E_j(n)\right)p_{n-1}(\lambda)\in L^2_\rho([-1,1]),
\end{align*}
recalling that $p_n$ are the orthonormal polynomials corresponding to $\rho$, and $\lambda=\cos x=(z+z^{-1})/2$.
Making use of the relationship \eqref{pnchin}, it will be sufficient to work on the torus, and compute limits for
\begin{align}\label{eq:spectral1}
O:=\sum_{j=0}^{\kappa-1} \widehat f(z_j)e^{-iT\cos x} \sum_{n\geq 0} E_j(n+1)s_n(z)
\end{align}
and
\begin{align}\label{eq:spectral2}
\tilde O:=\sum_{j=0}^{\kappa-1} \widehat f(z_j)e^{-iT\cos x} \sum_{n\geq 0} E_j(n+1)\overline{s_n}(z)
\end{align}
in $L^2_\sigma(\T)$ as $T\to\infty$.
%With the results of the preceding section, we will prove that 
%\begin{align}\label{mainasymptotic}
%O=e^{-iT\cos x}\sum_{j\in J}  \widehat f(z_j) \sum_{u=-\lceil n_j\rceil +1}^\infty E_j (\lceil n_j\rceil +u)\overline{\phi^*_{2\lceil n_j\rceil }}(z)z^{\lceil n_j\rceil +u}+o(1)\,
%\end{align}
%in $L^2_\sigma(\T)$ as $T\to \infty$. 
We first rewrite
\begin{align*}
O&=e^{-iT\cos x}\sum_{j=0}^{\kappa-1}  \widehat f(z_j) \sum_{u=-\lceil n_j\rceil }^\infty E_j(\lceil n_j\rceil+u+1)s_{\lceil n_j\rceil+u}(z)\\
&=e^{-iT\cos x}\sum_{j=0}^{\kappa-1}  \widehat f(z_j) \sum_{u=-\lceil n_j\rceil}^\infty \alpha_{j,u}s_{\lceil n_j\rceil+u}(z)
\end{align*}
for the sequence $\alpha_{j,u}:=E_j(\lceil n_j\rceil+u+1)$, which  satisfies
\begin{equation}\label{skor}
|\alpha_{j,u}|\stackrel{\eqref{eq:rapid}}{\leq_\ell}  \kappa^{-1}(1+|u|/\kappa)^{-\ell}
\end{equation}
for any $\ell\in \Z^+$.
Since
\begin{align*}
\|\sum_{j=0}^{\kappa-1}&  \widehat f(z_j) \sum_{u > \kappa^2} \alpha_{j,u}s_{\lceil n_j\rceil+u}(z)\|_{2,\sigma}\leq_\ell \|\widehat{f}\|_\infty \sum_{j=0}^{\kappa-1}\sum_{u>\kappa^2} \kappa^{-1}(1+|u|/\kappa)^{-\ell} \leq_\ell \kappa^{-\ell+2}\to 0
\end{align*}
as $\kappa\to \infty$ for $\ell>2$, we have
\begin{align}\label{trio}
O=e^{-iT\cos x}\sum_{j=0}^{\kappa-1}  \widehat f(z_j) \sum_{-\lceil n_j\rceil+1\le u \leq \kappa^2} \alpha_{j,u}s_{\lceil n_j\rceil+u}(z)+o(1),
\end{align}
as $T\to \infty$. Recall that $\widehat f=0$ in the $\delta$-neighborhoods of the points $\pm 1$ and that $n_j=-T\sin(x_j)=-T\sin(2\pi j/\kappa)$. Hence, $\widehat f(z_j)\neq 0$ implies  that
\[
|n_j|\ge |T|\sin(\delta)\geq (\lfloor \sqrt{T}\rfloor)^2\sin(\delta)=\kappa^2(\sin\delta)\,.
\]
The bound \eqref{skor} shows that the contribution in the sum in \eqref{trio} coming from the range $n_j\le-\kappa^2(\sin\delta)$ is negligible, so we are left to study the case $n_j\ge \kappa^2(\sin\delta)$. That restriction makes us read the values of $\widehat f(z)$ in only the lower half-circle, and in fact, only at $z_j$ with $x_j\in [\pi+\delta,2\pi -\delta]$. Moreover, our assumption that $\widehat f=0$ in the $\delta$-neighborhoods of the points $\pm i$ guarantees that the restriction on the indices given by $\widehat f(z_j)\neq 0$ implies that $|n_{j+1}-n_j|\sim_\delta\kappa$.
Hence, Corollary \ref{cor1} and the Remark \ref{rr} are applicable, and we get
\begin{align}\label{phasecancel}
O=\sum_{j=0}^{\kappa-1}  \widehat f(z_j) \overline{\phi^*_{2\lceil n_j\rceil }}(z)    \Bigl(e^{-iT\cos x}\sum_{u\in \Z} \alpha_{j,u}z^{u+\lceil n_j\rceil}\Bigr)+o(1)=\hspace{4cm}\\
\sum_{j=0}^{\kappa-1}  \widehat f(z_j) \overline{\phi^*_{2\lceil n_j\rceil }}(z)    \Bigl(e^{-iT\cos x}\sum_{s\in \Z} E_j(s+1)z^{s}\Bigr)+o(1)=z^{-1}\mathbb{1}_{\{\Im(z)<0\}}(z)\sum_{j=0}^{\kappa-1}  \widehat f(z_j) \overline{\phi^*_{2\lceil n_j\rceil }}(z) \omega_{j}(z)+o(1)\nonumber
\end{align}
in $L^2_\sigma(\T)$, by Fourier inversion, and using our observation that the sum is over only indices $j$ with $x_j\in [\pi+\delta,2\pi -\delta]$.

Now, by Theorem~\ref{lem:PCA} and Remark \ref{rr}, we finally have 
\begin{align}\label{Olimit}
O\to z^{-1}\mathbb{1}_{\{ \Im(z)<0\}}(z)\overline{\Pi(z)}\widehat{f}(z)
\end{align}
in $L^2_\sigma(\T)$. Using that $\overline{s_n}(z)=s_n(1/z)$ by the evenness of $\sigma$, we have $\tilde O(z)=O(1/z)$, so that  
\begin{align}\label{tildeOlimit}
\tilde O\to -z\mathbb{1}_{\{\Im(z)>0\}}\Pi(z)\widehat{f}(z)
\end{align}
since $\overline{\Pi}(z)=\Pi(1/z)$ and  $\widehat{f}$ is odd. With this in hand, we may examine the full eigenfunction expansion 
\begin{align*}
\sum_{n\geq 0}c_n (e^{iTJ_0^{\Z}}f)_{n+1}(s_n(z)+\overline{s_{n}}(z))\,.
\end{align*}
We recall that under the Szeg\H o condition, we have $c_n\to 1/\sqrt{2}$. Thus, using \eqref{eq:fourierapprox} and the support restriction on $\widehat f$, we have by Lemma \ref{lem:evobump}
\begin{align*}
\sum_{n\geq 0}c_n (e^{iTJ_0^{\Z}}f)_{n+1}(s_n(z)+\overline{s_{n}}(z))&=\frac{1}{\sqrt{2}}\sum_{n\geq  T\sin(\delta)/2 } (e^{iTJ_0^{\Z}}f)_{n+1}(s_n(z)+\overline{s_{n}}(z))+o(1)\\
&=\frac{1}{\sqrt{2}}\sum_{n\geq 0} (e^{iTJ_0^{\Z}}f)_{n+1}(s_n(z)+\overline{s_{n}}(z))+o(1), 
\end{align*}
as $T\to \infty$ in $L^2_\sigma(\T)$. Thus by \eqref{Olimit} and \eqref{tildeOlimit}, along with \eqref{eq:fourierapprox}, we have 
\begin{align*}
\lim_{T\to\infty}e^{-iT\cos(x)}\sum_{n\geq 0}c_n (e^{iTJ_0^{\Z}}f)_{n+1}(s_n(z)+\overline{s_{n}}(z))&=\frac{1}{\sqrt{2}}(z^{-1}\mathbb{1}_{\Im(z)<0}\overline{\Pi(z)}-z\mathbb{1}_{\Im(z)>0}\Pi(z))\widehat{f}(z)
\end{align*}
in $L^2_\sigma(\T)$, thus establishing the limit $\Omega_+$ on a dense set, which yields the strong limit \eqref{eq:WO}.\smallskip

We now show that the wave operators are complete. Denote
\[
m(z):=\frac{1}{\sqrt{2}}(z^{-1}\mathbb{1}_{\Im(z)<0}\overline{\Pi(z)}-z\mathbb{1}_{\Im(z)>0}\Pi(z)).
\]
Recall that $E_{\rm ac,\sigma}$ is the complement of the support of the singular part of $\sigma$ in $\T$. We define the operator
\begin{align*}
&T_m: L^2_{m,\rm odd}(\T)\to L^2_{\sigma'dm,\rm even}(E_{\rm ac,\sigma}),\quad
(T_mf)(\cdot)= m(\cdot)f(\cdot)
\end{align*}
which we note is well-defined since $\sigma'(z)=|D(z)|^2$ and $m(z)=-m(z^{-1})$, so that $mf$ is even for $f$ odd. $T_m$ is also invertible, with inverse given by multiplication by $1/m(z)$.

We denote by $U_0:\ell^2(\N)\to L^2_{m,{\rm odd}}(\T)$ the bijection taking $\{f_n\}$ to $\widehat{ f^{\rm (odd)}}$ for $\{f^{\rm (odd)}_n\}\in \ell^2(\Z)$, the odd extension as defined above. Recall also that  $U$ is defined in \eqref{spectralmap}. We have shown above that with $V:L^2_{\rho}([-1,1])\to L^2_{\sigma,\rm even}(\T)$ given by $(Vg)(z)=g(\frac{z+z^{-1}}{2})$,
\begin{align*}
VU\Omega_+U_0^{-1}=T_m,
\end{align*}
so that since $T_m$ is onto, 
\begin{align*}
{\rm Ran}(VU\Omega_+)=L^2_{\sigma'dm,\rm even}(E_{\rm ac,\sigma}).
\end{align*}
Denoting $\rho_{\rm ac}$ the density of the absolutely continuous part of $\rho$, we note that $V$ restricts to a bijection between $L^2_{\rho_{\rm ac}d\lambda}([-1,1])$ and $L^2_{\sigma'dm,even}(E_{\rm ac,\sigma})$, and we may conclude ${\rm Ran}(U\Omega_+)=L^2_{\rho_{\rm ac}d\lambda}([-1,1])$ so that ${\rm Ran}(\Omega_+)=\ell^2_{\rm ac}(\N)$ as required.

\end{proof}

We now describe conditions on the coefficients of $J$ that allow for the application of the above theorem, proving Corollary~\ref{cor:main}. 

\begin{proof}[Proof of Corollary~\ref{cor:main}]

By the Kato-Rosenblum theorem \cite[Theorem XI.8]{ReedSimon3}, we can assume that the norms of perturbations are as small as needed, and any $\ell^1(\N)$ perturbation is insignificant when proving existence and completeness of wave operators. Recall that if the Jacobi matrix is generated by the OPUC with parameters $\{\gamma_n\}\in \ell^2(\Z_+)$, then 
%\begin{align*}
%p_n=\frac{1}{\sqrt{2}}(1-\overline{\gamma}_{2n-1})^{-1/2}(s_n+%\bar s_n)
%\end{align*}
we have the direct Geronimus relations \eqref{geronimus}
\[
b_{k+1}=\frac 12 \left( (1-\gamma_{2k-1})(1-\gamma_{2k}^2)(1+\gamma_{2k+1})    \right)^{\frac 12}=\frac 12\left(1+\frac{\gamma_{2k+1}}{2}-\frac{\gamma_{2k-1}}{2}\right)+\ell^1(\N)
\]
and
\[
v_{k+1}=-\frac 12 (\gamma_{2k-2}(1+\gamma_{2k-1})-\gamma_{2k}(1-\gamma_{2k-1}))=\frac 12\left(\gamma_{2k}-\gamma_{2k-2}\right)+\ell^1(\N)\,.
\]
So, if $b_{k+1}=\frac 12(1+\beta_{k+1}-\beta_k)+\ell^1(\N)$ and $v_{k+1}=\frac 12(\alpha_{k+1}-\alpha_{k})$ with $\|\{\beta_k\}\|_{\ell^2}$ and $\|\{\alpha_k\}\|_{\ell^2}$ taken small enough, then we can define
$\gamma_{2n-2}=\alpha_n, \gamma_{2n-1}=2\beta_n$. Then, the Jacobi matrix generated by such a choice will differ from the given $J$ by an $\ell^1(\N)$ perturbation. The application of Theorem~\ref{thm:main} concludes the proof.

\end{proof}

%Finally, we may also prove Theorem~\ref{thm:Cesaro} and establish the existence of wave operators in Ces\`aro sense without the quantitative condition \eqref{loga}. 

%\begin{proof}
%Since the only place this condition appears in the proof of Theorem~\ref{thm:main} is in the application of Lemma~\ref{lem:PCA}, which in turn only depends on \eqref{loga} in the estimation of the integral $I_2$ in  \eqref{sum4}, we describe only this modification. 

%We claim that if $\gamma_n\in \ell^2$, then we have $I_2\to 0$ in the Ces\`aro sense. Indeed, following \eqref{lota}, let $F_j=\sum_{u\neq 0} u^{-1}e^{2\pi iju/\kappa}z^u$. Then, 
%\[
%I_2\le \sum_{j}\|{\rm Proj}_{H_j}F_j\|_{2,\sigma}^2
%\]
%and $H_j=\text{span}\{\phi_{\kappa^2+j\kappa},\ldots,\phi_{\kappa^2+j\kappa+\kappa}\}$. Changing the time from $\kappa^2$ to $\kappa^2+\kappa$ shifts the index of the space $H_j$ to $H_{j+1}$. Since 
%\[
%\sum_{p=0}^\kappa\|{\rm Proj}_{H_{j+p}}F_j\|_{2,\sigma}^2\lesssim 1.
%\]

%\end{proof}
\bigskip

\section{Appendix}
\subsection{Auxiliary proofs and OPUC estimate}\label{app:1}

\begin{lemma}\label{lem:aux} Assume that $\|\{\gamma_n\}\|_2\le \frac 12$. For every arc $I\subset \T$, let $\nu:=|I|^{-1}$. Then, for each $\ell\in \N$, we have
\begin{equation}
\int\limits_I |\phi_m|^2d\sigma\lesssim |I|+O_\ell\left(\sum_{n\ge 0} (n+1)^{-\ell}\sum_{|s-m|< \nu(n+1),s\ge 0} |\gamma_s|^2 \right)+O_\ell\Bigl((m/\nu)^{-\ell}\Bigr),\label{pi1}
\end{equation}
provided that $m\gtrsim  \nu$.
\end{lemma}
\begin{proof} Without loss of generality, we can assume that $I=(-1/(2\nu),1/(2\nu))$. Take $\rho\in C^\infty(\R)$ such that $\supp \,\rho\subset [-1,1]$ and $\rho=1$ on $[-1/2,1/2]$. Let $\omega(x)=\rho(x\nu)$ for $x\in (-1/\nu,1/\nu)$ and let it be zero for other $x\in [-\pi,\pi)$. Hence, we get
\[
|\widehat\omega_u|\le_\ell (\nu(1+|u|/\nu)^{\ell})^{-1}
\]
for every $\ell\in \Z^+$. Since $\{\phi_j\}$ is orthonormal,
\begin{equation}\label{kar1}
\|\sum_{j\ge 0} \phi_j \widehat\omega_{j-m}\|^2_{2,\sigma}=\sum_{j\ge 0} |\widehat\omega_{j-m}|^2\sim |I|+\nu^{-1}O_{\ell}\Bigl((m/\nu)^{-(2\ell-1)}\Bigr)
\end{equation}
if $\ell\ge 1$.
However,
\[
\sum_{j\ge 0} \phi_j \widehat\omega_{j-m}=z^m\sum_{j\ge 0} \overline{\phi_j^*}z^{j-m} \widehat\omega_{j-m}=z^m\sum_{j\ge 0} (\overline{\phi_j^*}-\overline{\phi_m^*})z^{j-m} \widehat\omega_{j-m}+z^m\overline{\phi_m^*}\sum_{j\ge 0} z^{j-m} \widehat\omega_{j-m}=:I_1+I_2\,.
\]
Notice that
\begin{equation}\label{kar2}
\|I_2-z^m\overline{\phi_m^*}\omega\|_{2,\sigma}=O_\ell\Bigl((m/\nu)^{-(\ell-1)}\Bigr)
\end{equation}
if $\ell\ge 2$. 
For $I_1$, the bound is
\[
\|I_1\|_{2,\sigma}\le_\ell\sum_{n\ge 0} \left(\sum_{|s-m|< \nu(n+1),s\ge 0} |\gamma_s|^2\right)^{\frac 12} (n+1)^{-\ell}\,.
\]
By the Cauchy-Schwarz bound, we get
\begin{equation}\label{kar3}
\|I_1\|^2_{2,\sigma}\le_\ell\sum_{n\ge 0} \left(\sum_{|s-m|< \nu(n+1),s\ge 0} |\gamma_s|^2\right) (n+1)^{-2\ell}\,.
\end{equation}
We need to combine \eqref{kar1}, \eqref{kar2}, \eqref{kar3} and apply triangle inequality to finish the proof.
\end{proof}

We continue  by proving the corollary stated in Section~\ref{section2} above. 

\begin{proof}[Proof of Corollary~\ref{cor1}]
Take any $P\in \N$ and represent $f_{j,\kappa}=\sum_{d=0}^{P-1} f_{j,\kappa}^{(d)}$ according to $j\equiv d\mod P$. 
Write $A=\sum_{d=0}^{P-1}A^{(d)}$ and $B=\sum_{d=0}^{P-1}B^{(d)}$, respectively. We claim that
\begin{equation}\label{gap1}
\|A^{(d)}-B^{(d)}\|_{2,\sigma}=o(1)+O(P^{-2})
\end{equation}
and 
\begin{equation}\label{pol-53}
 \sum_{j=0}^{\kappa-1} |f^{(d)}_{j,\kappa}|^2
\|\phi^*_{2(\kappa^2+2j\kappa)}\omega_j\|^2_{2,\sigma} 
-\int\limits_{\mathbb T} \sum_{j=0}^{\kappa-1} |f^{(d)}_{j,\kappa}|^2|\omega_j(z)|^2dm
=o(1)+O(P^{-2})
\end{equation}
as $\kappa\to\infty$. 

 Without loss of generality, take $d=0$. Recall that $\omega_0(z)=h(\kappa x), z=e^{ix}$, see \eqref{base_1}. Our $h$ is smooth and compactly supported on $\R$ so $\widehat \omega_0$ is not compactly supported, and we cannot use Lemma~\ref{lem1} or Lemma~\ref{post1} directly. Instead, we will use an approximation argument. Take an even function $\mu(\xi)\in C^\infty_c(\R)$ such that $\mu=1$ on $[-1,1]$ and $\mu=0$ for $|\xi|>2$ and let $h_m(x)=\mathcal{F}^{-1}((\mathcal{F} h)(\xi)\mu(\xi/m))$
 where $m\in \N$ is a large parameter to be chosen later. Clearly, $(\mathcal F h_m)(\xi)=(\mathcal F h)(\xi)$ for $|\xi|<m$,  $(\mathcal F h_m)(\xi)=0$ for $|\xi|>2m$, and $(\mathcal F h_m(\kappa x))(\xi)=\kappa^{-1}(\mathcal F h_m(x))(\xi/\kappa)$. Next, we let 
 \[
\tau_0(e^{ix})=\sum_{p\in \Z}h_m((x-2\pi p)\kappa), \quad x\in \R\,.
 \]
 This is a smooth function on $\T$. By the Poisson summation formula, $(\widehat\tau_0)_n=(\mathcal{F}h_m(x\kappa))(n)=\kappa^{-1}(\mathcal F h_m(x))(n/\kappa)=\kappa^{-1}(\mathcal F h)(n/\kappa)\mu(n/(\kappa m))$. Hence,\smallskip
 
 (A) we have $(\widehat\tau_0)_n=(\widehat\omega_0)_n$ for $|n|<\kappa m$ and $(\widehat\tau_0)_n=0$ for $|n|>2\kappa m$,\smallskip
 
 (B) the function $\tau_0(z)$ is $\infty$-localized to the arc $I_0$ uniformly in $m$. \smallskip
 
 (C) the bound 
 \begin{equation}\label{com-1}
 |(\widehat\tau_0)_n-(\widehat\omega_0)_n|\le_\ell  m^{-\ell}(\kappa(1+|n/\kappa|^\ell))^{-1} 
 \end{equation}
 holds for any $n\in \Z, \ell\in \Z_+$.\smallskip
 
 Now, we introduce 
 \[
 \tau_j(z):=\tau_0(z/z_j), \quad \delta_j(z):=\omega_j(z)-\tau_j(z)
 \]
 and write
 \[
 \tau_j(z)=:\sum_{u\in \Z} \eta_{j,u}z^u, \quad  \delta_j(z)=:\sum_{u\in \Z} \beta_{j,u}z^u\,.
 \]
 Notice that both $\tau_j$ and $\delta_j$ are $\infty$-localized at $I_j$ and 
 \begin{equation}\label{pes5}
 \|\delta_j\|_{L^\infty(\T)}\stackrel{\eqref{com-1}}{\le_\ell} m^{-\ell}.
 \end{equation}
 Then,
 \[
 A^{(0)}=\sum_{j=0}^{\kappa-1}\sum_{|u|\le\kappa^2}f^{(0)}_{j,\kappa}s_{\kappa^2+2j\kappa+u}\alpha_{j,u}=I_1+I_2\,,
 \]
 where
 \[
 I_1:= \sum_{j=0}^{\kappa-1}\sum_{|u|\le\kappa^2}f^{(0)}_{j,\kappa}s_{\kappa^2+2j\kappa+u}\eta_{j,u}, \,  I_2:=\sum_{j=0}^{\kappa-1}\sum_{|u|\le\kappa^2}f^{(0)}_{j,\kappa}s_{\kappa^2+2j\kappa+u}\beta_{j,u}\,.
 \]
 Similarly, 
 \[
 B^{(0)}=\sum_{j=0}^{\kappa-1}
f^{(0)}_{j,\kappa}s_{\kappa^2+2j\kappa}\sum_{u\in\Z}\alpha_{j,u}z^u=J_1+J_2\,,
 \]
 where
 \[
 J_1:=\sum_{j=0}^{\kappa-1}
f^{(0)}_{j,\kappa}s_{\kappa^2+2j\kappa}\sum_{u\in\Z}\eta_{j,u}z^u, \quad J_2:=\sum_{j=0}^{\kappa-1}
f^{(0)}_{j,\kappa}s_{\kappa^2+2j\kappa}\sum_{u\in\Z}\beta_{j,u}z^u=\sum_{j=0}^{\kappa-1}
f^{(0)}_{j,\kappa}s_{\kappa^2+2j\kappa}\delta_j\,.
 \]
Taking $m=\lfloor P/4\rfloor$, we get $\|I_1-J_1\|_{2,\sigma}\to 0$ as $\kappa\to\infty$ by adjusting the proof of \eqref{step-11}.
 We are going to show that
 \begin{equation}\label{fir3}
 \|I_2\|_{2,\sigma}\le_\ell P^{-\ell}, \quad  \|J_2\|_{2,\sigma}\le_\ell P^{-\ell/4}\,.
 \end{equation}
 Partitioning $\T=\cup_{s=0}^{\kappa-1}I_s$, we may separate diagonal and off-diagonal terms and write
  \begin{align*}
\|J_2\|^2_{2,\sigma}&{\leq}_{\ell} \,\mathcal{D}+\sum_{s=0}^{\kappa-1}\sum_{j_1\ne j_2}|f^{(0)}_{j_1,\kappa}||f^{(0)}_{j_2,\kappa}|\frac{\left(\int\limits_{I_s}|\phi^*_{2(\kappa^2+2j_1\kappa)}|^2d\sigma(z)\right)^{1/2}\left( \int\limits_{I_s}|\phi^*_{2(\kappa^2+2j_2\kappa)}|^2d\sigma(z)\right)^{1/2}}{(|j_1-s|^\ell+1)(|j_2-s|^\ell+1)} \\
&\leq_{\ell}\mathcal{D}+\sum_{s=0}^{\kappa-1}\sum_{j_1\ne j_2}|f^{(0)}_{j_1,\kappa}||f^{(0)}_{j_2,\kappa}|\frac{\left(\int\limits_{I_s}|\phi^*_{2(\kappa^2+2j_1\kappa)}|^2d\sigma(z)\right)^{1/2}\left( \int\limits_{I_s}|\phi^*_{2(\kappa^2+2j_2\kappa)}|^2d\sigma(z)\right)^{1/2}}{|j_1-j_2|^{\ell/2}(|j_1-s|^{\ell/2}+1)(|j_2-s|^{\ell/2}+1)}  \\
&\leq_{\ell} \mathcal{D}+P^{-\ell/2}\sum_{s=0}^{\kappa-1}\sum_{j_1\ne j_2}\frac{ \left(\int\limits_{I_s}|\phi^*_{2(\kappa^2+2j_1\kappa)}|^2d\sigma(z)\right)^{1/2}\left( \int\limits_{I_s}|\phi^*_{2(\kappa^2+2j_2\kappa)}|^2d\sigma(z)\right)^{1/2}}{(|j_1-s|^{\ell/2}+1)(|j_2-s|^{\ell/2}+1)}
 \end{align*}
by $\infty$-localization of $\delta_j$, and where $\mathcal D$ collects the diagonal terms:
 \begin{align*}
 \mathcal D:=\sum_{j=0}^{\kappa-1}\sum_{s=0}^{\kappa-1}|f_{j,\kappa}^{(0)}|^2\int\limits_{I_s}|\phi_{2(\kappa^2+2j\kappa)}^*|^2|\delta_j|^2d\sigma(z).
 \end{align*}
We estimate the off-diagonal terms first.  For these terms, after applying the inequality $2xy\le x^2+y^2$ in the numerator and summing out the $j_{1(2)}$ factor (assuming $\ell\ge 4$), it suffices to show 
\begin{align}\label{last-09}
\sum_{s=0}^{\kappa-1}\sum_{j=0}^{\kappa-1}\frac{1}{|j-s|^{\ell/2}+1} \int\limits_{I_s}|\phi^*_{2(\kappa^2+2j\kappa)}|^2d\sigma(z)=O(1)\,.
\end{align}
We now make use of Lemma~\ref{lem:aux},  set $m_j=2(\kappa^2 +2j\kappa)$ and define
\begin{align*}
\sum_{n\ge 0} (n+1)^{-\ell}\sum_{|r-m_j|< \kappa(n+1)} |\gamma_r|^2=:E_j
\end{align*}
so \eqref{pi1} gives
\begin{align*}
\int\limits_{I_s}|\phi^*_{2(\kappa^2+2j\kappa)}|^2d\sigma(z)\leq E_j+O(1/\kappa)\,.
\end{align*}
Taking  $\ell\ge 4$, we may sum out the $s$-index and find
\begin{align*}
\sum_{s=0}^{\kappa-1}\sum_{j=0}^{\kappa-1}\frac{1}{|j-s|^{\ell/2}+1} \int\limits_{I_s}|\phi^*_{2(\kappa^2+2j\kappa)}|^2d\sigma(z)&\leq_\ell \sum_{j=0}^{\kappa-1}E_j+O(1)\\
&=\sum_{n\ge 0} (n+1)^{-\ell}\sum_{j=0}^{\kappa-1}\sum_{|r-m_j|< \kappa(n+1)} |\gamma_r|^2+O(1)\,.
\end{align*}
We note that 
\begin{align*}
\sum_{j=0}^{\kappa-1}\sum_{|r-m_j|< \kappa(n+1)} |\gamma_r|^2\lesssim  (n+1)\sum_{r\geq 0}|\gamma_r|^2,
\end{align*}
 so that
 \begin{align*}
 \sum_{j=0}^{\kappa-1}E_j\lesssim \sum_{n\geq 0}(n+1)^{-\ell+1},
 \end{align*}
establishing \eqref{last-09}. We estimate $\mathcal{D}$ by separating the terms within a $P$-width of the diagonal as

\begin{eqnarray}\label{pol-56}
 \mathcal{D}\le_{\ell}\sum_{j=0}^{\kappa-1}\sum_{s=0}^{\kappa-1} \int\limits_{I_s}|\phi^*_{2(\kappa^2+2j\kappa)}|^2|\delta_j|^2d\sigma(z)\hspace{7cm}\\
{\le_\ell}
\sum_{|s-j|<P}\int\limits_{I_s}|\phi^*_{2(\kappa^2+2j\kappa)}|^2|\delta_j|^2d\sigma(z)+
\sum_{|s-j|\ge P}\frac{1}{|j-s|^{2\ell}+1} \int\limits_{I_s}|\phi^*_{2(\kappa^2+2j\kappa)}|^2d\sigma(z)\hspace{1cm} \nonumber\\
\stackrel{\eqref{pes5}}{\le_{\ell}}\sum_{|s-j|<P}P^{-2\ell}\frac{|j-s|^{\ell}+1}{|j-s|^{\ell}+1} \int\limits_{I_s}|\phi^*_{2(\kappa^2+2j\kappa)}|^2d\sigma(z)+
\sum_{|s-j|\ge P}\frac{1}{|j-s|^{2\ell}+1} \int\limits_{I_s}|\phi^*_{2(\kappa^2+2j\kappa)}|^2d\sigma(z) \nonumber\\
\le_{\ell}\sum_{|s-j|<P}\frac{P^{-\ell}}{|j-s|^{\ell}+1} \int\limits_{I_s}|\phi^*_{2(\kappa^2+2j\kappa)}|^2d\sigma(z)+
\sum_{|s-j|\ge P}\frac{P^{-\ell}}{|j-s|^{\ell}+1} \int\limits_{I_s}|\phi^*_{2(\kappa^2+2j\kappa)}|^2d\sigma(z)\stackrel{\eqref{last-09}}{\le_\ell}P^{-\ell}\,.\nonumber
\end{eqnarray}
Thus, the second bound in \eqref{fir3} is established.

%
%\begin{align*}
%\mathcal{D}&=\sum_{j=0}^{\kappa-1}|f_{j,\kappa}|^2\int\limits_{I_j}|\phi_{2(\kappa^2+2j\kappa)}^*|^2|\delta_{j}|^2d\sigma(z)+\sum_{1\leq |j-s|\leq P}|f_{j,\kappa}|^2\int\limits_{I_s}|\phi_{2(\kappa^2+2j\kappa)}^*|^2|\delta_{j}|^2d\sigma(z)\\
%&+\sum_{|j-s|>P}|f_{j,\kappa}|^2\int\limits_{I_s}|\phi_{2(\kappa^2+2j\kappa)}^*|^2|\delta_{j}|^2d\sigma(z).
%\end{align*}
%We use \eqref{pes5} on the first and second term, and the $\infty$-localization of $\delta_j$ on the final term to find
%\begin{align*}
%\mathcal{D}&\leq_{\ell}P^{-2\ell}\sum_{j=0}^{\kappa-1|}|f_{j,\kappa}|^2\int\limits_{I_j}|\phi_{2(\kappa^2+2j\kappa)}^*|^2d\sigma(z)+P^{-2\ell}\sum_{1\leq |j-s|\leq P}|f_{j,\kappa}|^2\int\limits_{I_s}|\phi_{2(\kappa^2+2j\kappa)}^*|^2d\sigma(z)\\
%&+P^{-\ell}\sum_{|j-s|>P}|f_{j,\kappa}|^2\frac{1}{1+|j-s|^{2\ell}}\int\limits_{I_s}|\phi_{2(\kappa^2+2j\kappa)}^*|^2d\sigma(z)\\
%&\leq_{\ell}P^{-\ell}\left( \sum_{j=0}^{\kappa-1|}\int\limits_{I_j}|\phi_{2(\kappa^2+2j\kappa)}^*|^2d\sigma(z)+\sum_{1\leq |j-s|\leq P}\int\limits_{I_s}|\phi_{2(\kappa^2+2j\kappa)}^*|^2d\sigma(z)+\sum_{|j-s|>P}\frac{\int_{I_s}|\phi_{2(\kappa^2+2j\kappa)}^*|^2}{1+|j-s|^{\ell}}d\sigma(z)\right)
%\end{align*}
%pulling out a factor of $P^{-\ell}$ from the final term, and using $|f_{j,\kappa}|\lesssim 1$. To see that the three terms in parentheses are $O(1)$, we use \eqref{last-09} for $\ell\geq 2$ for the final term, and \eqref{eq:Ej} for the first two terms. 

Next, we write
\[
I_2=\sum_{j=0}^{\kappa-1}\sum_{|u|\le\kappa^2}f^{(0)}_{j,\kappa}s_{\kappa^2+2j\kappa+u}\beta_{j,u}=\sum_{n=0}^{5\kappa^2}s_nc_n, \quad c_n=\sum_{j,u: \,\kappa^2+2j\kappa+u=n} f_{j,\kappa}^{(0)}\beta_{j,u}
\]
and \eqref{com-1} gives $|c_n|\le_{\ell}P^{-\ell}\kappa^{-1}$. The orthogonality of the system $\{s_n\}$ yields
$
\|I_2\|_{2,\sigma}\le_{\ell} P^{-\ell}
$
and we have \eqref{gap1} for $\ell\ge 8$.
 \medskip
 
To prove \eqref{pol-53}, we argue similarly by substituting $\omega_j=\tau_j+\delta_j$ and using the triangle inequality. First, one can claim
 \begin{equation}\label{pol-54}
 \sum_{j=0}^{\kappa-1} |f^{(0)}_{j,\kappa}|^2
\|\phi^*_{2(\kappa^2+2j\kappa)}\tau_j\|^2_{2,\sigma} 
-\int\limits_{\mathbb T} \sum_{j=0}^{\kappa-1} |f^{(0)}_{j,\kappa}|^2|\tau_j(z)|^2dm=o(1), \quad \kappa\to\infty
\end{equation}
 by adjusting the proof of \eqref{lisa3}. Secondly, after decomposing the torus $\T=\cup_{s=0}^{\kappa-1}I_s$ again, the $\infty$-localization of $\delta_j$ around $I_j$ and \eqref{pes5} give
 \begin{equation}\label{pol-55}
 \int\limits_{\mathbb T} \sum_{j=0}^{\kappa-1} |f^{(0)}_{j,\kappa}|^2|\delta_j(z)|^2dm\stackrel{\eqref{pes5}}{\le_\ell} \kappa^{-1}\sum_{j=0}^{\kappa-1}\sum_{s=0}^{\kappa-1}\min\left(P^{-2\ell}, (|s-j|^{2\ell}+1)^{-1}\right)\le_{\ell}P^{-\ell}\,.
 \end{equation}
 Finally, since $|f_{j,\kappa}|\lesssim 1$, we may argue exactly as in \eqref{pol-56} and find
 \begin{align}\label{eq:remainder}
   \sum_{j=0}^{\kappa-1} |f^{(0)}_{j,\kappa}|^2
\|\phi^*_{2(\kappa^2+2j\kappa)}\delta_j\|^2_{2,\sigma}\le_{\ell}\sum_{j=0}^{\kappa-1}\sum_{s=0}^{\kappa-1} \int\limits_{I_s}|\phi^*_{2(\kappa^2+2j\kappa)}|^2|\delta_j|^2d\sigma(z)\leq_{\ell}P^{-\ell}.
 \end{align}
%
% \begin{eqnarray}\label{pol-56}
%  \sum_{j=0}^{\kappa-1} |f^{(0)}_{j,\kappa}|^2
%\|\phi^*_{2(\kappa^2+2j\kappa)}\delta_j\|^2_{2,\sigma}\le_{\ell}\sum_{j=0}^{\kappa-1}\sum_{s=0}^{\kappa-1} \int\limits_{I_s}|\phi^*_{2(\kappa^2+2j\kappa)}|^2|\delta_j|^2d\sigma(z)\\
%{\le_\ell}
%\sum_{|s-j|<P}\int\limits_{I_s}|\phi^*_{2(\kappa^2+2j\kappa)}|^2|\delta_j|^2d\sigma(z)+
%\sum_{|s-j|\ge P}\frac{1}{|j-s|^{2\ell}+1} \int\limits_{I_s}|\phi^*_{2(\kappa^2+2j\kappa)}|^2d\sigma(z)\nonumber\\
%\stackrel{\eqref{pes5}}{\le_{\ell}}\sum_{|s-j|<P}P^{-2\ell}\frac{|j-s|^{\ell}+1}{|j-s|^{\ell}+1} \int\limits_{I_s}|\phi^*_{2(\kappa^2+2j\kappa)}|^2d\sigma(z)+
%\sum_{|s-j|\ge P}\frac{1}{|j-s|^{2\ell}+1} \int\limits_{I_s}|\phi^*_{2(\kappa^2+2j\kappa)}|^2d\sigma(z)\nonumber\\
%\le_{\ell}\sum_{|s-j|<P}\frac{P^{-\ell}}{|j-s|^{\ell}+1} \int\limits_{I_s}|\phi^*_{2(\kappa^2+2j\kappa)}|^2d\sigma(z)+
%\sum_{|s-j|\ge P}\frac{P^{-\ell}}{|j-s|^{\ell}+1} \int\limits_{I_s}|\phi^*_{2(\kappa^2+2j\kappa)}|^2d\sigma(z)\stackrel{\eqref{last-09}}{\le_\ell}P^{-\ell}\,.\nonumber
%\end{eqnarray}
 Now, \eqref{pol-53} follows from \eqref{pol-54}, \eqref{pol-55}, and \eqref{eq:remainder} after we take large $\ell$ and apply the triangle inequality in the space $\{h_j(x)\}|_{j=0}^{\kappa-1}$ with the norm
 $
 \|h\|_1:=(\int\limits_{\T}\sum_{j=0}^{\kappa-1}|h_j(x)|^2d\sigma)^{\frac 12}
 $ for the first term and with the norm $
\|h\|_2:= (\int\limits_{\T}\sum_{j=0}^{\kappa-1}|h_j(x)|^2dm)^{\frac 12}
 $
 for the second as follows. Take $X:=\{f_{j,\kappa}^{(0)}\phi^*_{2(\kappa^2+2j\kappa)}\tau_j\}|_{j=0}^{\kappa-1}$, $Y:=\{f_{j,\kappa}^{(0)}\phi^*_{2(\kappa^2+2j\kappa)}\delta_j\}|_{j=0}^{\kappa-1}$ and $U:=\{f_{j,\kappa}^{(0)}\tau_j\}|_{j=0}^{\kappa-1}$, $W:=\{f_{j,\kappa}^{(0)}\delta_j\}|_{j=0}^{\kappa-1}$. One has
 \begin{equation}\label{gur3}
 \|X+Y\|_1^2-\|U+W\|_2^2=(\|X+Y\|_1-\|U+W\|_2)(\|X+Y\|_1+\|U+W\|_2)\,.
 \end{equation}
 Then,
 \[
 \|X+Y\|_1\stackrel{\eqref{eq:remainder}}{=}\|X\|_1+O(P^{-\ell/2})
 \]
 by the triangle inequality and, similarly, 
 \[
 \|U+W\|_2\stackrel{\eqref{pol-55}}{=}\|U\|_2+O(P^{-\ell/2})\,.
 \]
 We also have
 \[
 \|X\|_1^2\stackrel{\eqref{pol-54}}{=}\|U\|_2^2+o(1)\,,
 \]
 where $\|U\|_2\lesssim 1$ by the $\infty$-localization of $\{\tau_j\}$. Hence, $\|X\|_1\lesssim 1$ and, substitution into \eqref{gur3} gives the required bound, and we may conclude \eqref{pol-53}.

 Given \eqref{gap1}, we apply the triangle inequality to write
\[
\|A-B\|_{2,\sigma}\le \sum_{d=0}^{P-1} \|A^{(d)}-B^{(d)}\|_{2,\sigma}
\]
so
\[
\limsup_{\kappa\to\infty}\|A-B\|_{2,\sigma}\le \sum_{d=0}^{P-1} \limsup_{\kappa\to\infty}\|A^{(d)}-B^{(d)}\|_{2,\sigma}\le_{\ell}P^{-1}\,.
\]
Sending $P\to\infty$, we get \eqref{AB0}. In a similar way, one shows that \eqref{pol-53} yields \eqref{pol}. \end{proof}

%We record the following claim on sufficiency of proving existence of wave operators on an integer sequence.
%
%\begin{lemma}
%Suppose $\{\gamma_n\}\in \ell^2$ and \eqref{loga} holds. Then
%\[
%Wf=\lim_{n\to\infty}e^{inJ}e^{-inJ_0}f
%\]
%for every $f\in H$ implies 
%\begin{equation}\label{l}
%\lim_{t\to\infty}e^{itJ}e^{-itJ_0}f=Wf
%\end{equation}
%\end{lemma}
%
%\begin{proof}
%Claim. Suppose $\{\delta_n\}\in [0,1)$ and $\delta_n\to\delta^*$. First, we notice that 
%\[
%\lim_{n\to\infty}e^{i(n+\delta_n)J}e^{-i(n+\delta_n)J_0}f
%\]
%exists and is equal to $Wf$. Indeed, this is because $e^{-i\delta_n J_0}f\to e^{-i\delta^*J_0}f$ and $e^{i\delta_n J}g\to e^{i\delta^*J}g$ for all $f,g$. Moreover, $e^{-i\delta^*J}We^{i\delta^* J_0}f=Wf$ as follows from the spectral representation for $Wf$ we obtained in our proof, see \eqref{phasecancel}.
%
%Next, we argue by contradiction. Suppose \eqref{l} fails. That means, there is $t_n=m_n+\delta_n$ such that $m_n\in \mathbb{N}, m_n\to\infty,\delta_n\in [0,1)$ and
%\[
%\|e^{it_nJ}e^{-it_nJ_0}f-Wf\|\ge \epsilon>0
%\]
%for all $n$. We can find subsequence $\{k_n\}$ such that $\delta_{k_n}\to\delta^*$, $m_{k_n}$ is monotonic, and then the previous claim gives a contradiction.
%
%\end{proof}
%
\subsection{The pointwise convergence assumption }\label{app:PCA}

As a motivation for the condition \eqref{loga} above, we introduce the following definition. \smallskip

\noindent {\bf Definition.} Consider a measure $\sigma$ in the Szeg\H o class. We say that the pointwise convergence assumption (PCA) holds if $\{\phi_n^*(z)\}$ converges for a.e. $z\in \T$.\medskip

Recall that, under the Szeg\H o condition, $\phi^*_n\to \Pi$ in $L^2_\sigma(\T)$ where $\Pi=D^{-1}\cdot \mathbb{1}_{E_{{\rm ac},\sigma}}$. Then, the PCA implies  that $\{\phi_n^*(z)\}\to \Pi$ Lebesgue a.e.. By the Menshov-Rademacher theorem (see \cite{Denisov2025}), the condition $\gamma_n\log (n+1)\in \ell^2(\Z^+)$ implies the PCA, but it is not known if the PCA holds for all Szeg\H{o} measures. Clearly, assumption \eqref{loga} is weaker than  $\gamma_n\log (n+1)\in \ell^2(\Z^+)$.\smallskip

More directly related to our methods, we prove that the input to the proof of Theorem~\ref{thm:main} that relies on \eqref{loga} is also implied by the PCA.\medskip

\begin{lemma}The PCA implies \eqref{main}.
\end{lemma}
\begin{proof} 
By Egorov's theorem, we know that for each $\epsilon>0$, there is a set $E_\epsilon\subset \T$ such that $|E_\epsilon^c|\le \epsilon$ and $\phi_n^*-\Pi\to 0$ uniformly for $z\in E_\epsilon$. 
In estimating $\left\|\Sigma_{\kappa}-\Pi\sum_{j=0}^{\kappa-1} f_{j,\kappa}\omega_j\right\|_{2,\sigma}$, we use $\omega_{j_1}\omega_{j_2}=0$ for $\dist_{\T}(z_{j_1},z_{j_2})>2\pi/\kappa$ to get
\[
\left\|\Sigma_{\kappa}-\Pi\sum_{j=0}^{\kappa-1} f_{j,\kappa}\omega_j\right\|_{2,\sigma}^2\lesssim \int\limits_{\mathbb T} \sum_{j=0}^{\kappa-1} |\phi^*_{\kappa^2+j\kappa}-\Pi|^2|f_{j,\kappa}|^2\omega_j^2d\sigma\,.
\]
From \eqref{pol} and Remark \ref{rr}, we have
\begin{eqnarray*}
\int\limits_{E_\epsilon^c} \sum_{j=0}^{\kappa-1} |f_{j,\kappa}|^2 |\phi_{\kappa^2+j\kappa}(z)|^2 \omega_j(z)^2\sigma'dm+
\int\limits_{E_\epsilon} \sum_{j=0}^{\kappa-1} |f_{j,\kappa}|^2 |\phi_{\kappa^2+j\kappa}(z)|^2 \omega_j(z)^2\sigma'dm \\
+\int\limits_{\mathbb T}\sum_{j=0}^{\kappa-1}  |f_{j,\kappa}|^2|\phi_{\kappa^2+j\kappa}(z)|^2 \omega_j(z)^2d\sigma_s
-\int\limits_{\mathbb T} \sum_{j=0}^{\kappa-1} |f_{j,\kappa}|^2\omega_j(z)^2dm=o(1), \quad \kappa\to\infty\,.
\end{eqnarray*}

Since 
$|\Pi|^2\sigma'dm=dm$, we have, by uniform convergence 
\begin{align*}
\lim_{\kappa\to\infty}\int\limits_{E_{\epsilon}}\sum_{j=0}^{\kappa-1}|f_{j,\kappa}|^2|\phi_{\kappa^2+j\kappa}|^2\omega_j(z)^2\sigma'(z)dm(z)
-\int\limits_{E_{\epsilon}}\sum_{j=0}^{\kappa-1}|f_{j,\kappa}|^2\omega_j(z)^2dm(z)=0\,.
\end{align*}
One then has
\begin{align}\nonumber
\limsup_{\kappa\to\infty}&\left(\int\limits_{E_\epsilon^c} \sum_{j=0}^{\kappa-1} |f_{j,\kappa}|^2  |\phi_{\kappa^2+j\kappa}(z)|^2 \omega_j(z)^2\sigma'dm
+\int\limits_{\mathbb T}\sum_{j=0}^{\kappa-1} |f_{j,\kappa}|^2 |\phi_{\kappa^2+j\kappa}(z)|^2 \omega_j(z)^2d\sigma_s
\right)\\
&\le \limsup_{\kappa\to\infty}\int\limits_{E_\epsilon^c}\sum_{j=0}^{\kappa-1}|f_{j,k}|^2\omega_j(z)^2dm(z)
\le\epsilon\,,\label{loi1}
\end{align}
because $|f_{j,\kappa}|\lesssim  1$ and $\sum_{j=0}^{\kappa-1}\omega_j=1$. Since $\epsilon$ was arbitrary, we get
\begin{equation}\label{po2}
\lim_{\kappa\to\infty}\int\limits_{\mathbb T}\sum_{j=0}^{\kappa-1} |f_{j,\kappa}|^2 |\phi_{\kappa^2+j\kappa}(z)|^2 \omega_j(z)^2d\sigma_s=0\,.
\end{equation}
We continue by writing
\begin{eqnarray*}
\limsup_{\kappa\to\infty}\int\limits_{\mathbb T} \sum_{j=0}^{\kappa-1} |\phi^*_{\kappa^2+j\kappa}-\Pi|^2|f_{j,\kappa}|^2\omega_j^2\sigma'dm\le \\\limsup_{\kappa\to\infty}\int\limits_{E_\epsilon} \sum_{j=0}^{\kappa-1} |\phi^*_{\kappa^2+j\kappa}-\Pi|^2|f_{j,\kappa}|^2\omega_j^2\sigma'dm+\limsup_{\kappa\to\infty}\int\limits_{E^c_{\epsilon}} \sum_{j=0}^{\kappa-1} |\phi^*_{\kappa^2+j\kappa}-\Pi|^2|f_{j,\kappa}|^2\omega_j^2\sigma'dm\,.
\end{eqnarray*}
The first term is zero by the uniform convergence. For the second one, we have
\begin{eqnarray*}
\limsup_{\kappa\to\infty}\int\limits_{E^c_\epsilon} \sum_{j=0}^{\kappa-1} |\phi^*_{\kappa^2+j\kappa}-\Pi|^2|f_{j,\kappa}|^2\omega_j^2\sigma'dm\le\\ 2\limsup_{\kappa\to\infty}\int\limits_{E^c_\epsilon} \sum_{j=0}^{\kappa-1} (|\phi^*_{\kappa^2+j\kappa}|^2+|\Pi|^2)|f_{j,\kappa}|^2\omega_j^2\sigma'dm\stackrel{\eqref{loi1}}{\lesssim} \epsilon\,.
\end{eqnarray*}
Since $\epsilon$ is arbitrary, we get our result.
\end{proof}
\begin{remark}Following our Remark \ref{rr}, the previous lemma shows that the PCA implies existence and completeness of $\Omega_{\pm}$ as defined in \eqref{eq:WO}, since the assumption \eqref{loga} only enters the proof of Theorem~\ref{thm:main} through Theorem~\ref{lem:PCA}. In fact, we note that the convergence result \eqref{main} may be viewed as the key input beyond the Szeg\H o condition we need to prove Theorem~\ref{thm:main}.
\end{remark}

\end{document}